\documentclass{svjour3}                     
\smartqed  
\usepackage{graphicx}
\usepackage{mathptmx}      
\usepackage{color}
\usepackage{amsmath}
\usepackage{mathrsfs, amssymb}
\usepackage{epsfig}
\usepackage[colorlinks,citecolor=blue,urlcolor=blue,linkcolor=blue]{hyperref}

\newtheorem{thm}{Theorem}

\newtheorem{lem}[thm]{Lemma}

\newtheorem{rem}[thm]{Remark}

\newcommand{\var}{Var}   
\newcommand{\tr}{tr}         
      
        \newcommand{\E}{E}

\newcommand{\la}{\langle}

\renewcommand{\(}{\left(}
\renewcommand{\)}{\right)}
\newcommand{\lb}{\left[}
\newcommand{\rb}{\right]}
\newcommand{\lj}{\left|}
\newcommand{\rj}{\right|}
\newcommand{\De}{\Delta}
\renewcommand{\la}{\lambda}

\newcommand{\de}{\delta}
\newcommand{\al}{\alpha}

\newcommand{\ga}{\gamma}

\newcommand{\ep}{\epsilon}

\newcommand{\fLa}{{\pmb \Lambda}}

\newcommand{\fA}{{\bf A}}

\newcommand{\fM}{{\bf M}}

\newcommand{\fU}{{\bf U}}

\newcommand{\fW}{{\bf W}}
\newcommand{\fX}{{\bf X}}

\newcommand{\fe}{{\bf e}}
\newcommand{\fx}{{\bf x}}

\newcommand{\fw}{{\bf w}}
\newcommand{\fz}{{\bf z}}
\newcommand{\fy}{{\bf y}}

\newcommand{\cD}{{\mathcal D}}

\newcommand{\bC}{{\mathbb C}}

\def\indic{{\mathbf{I}}}

\def\to{\rightarrow}

\begin{document}
	\title{Convergence Rate of Eigenvector Empirical Spectral Distribution of Large Wigner Matrices
	}
	
	
	
	\author{Ningning Xia         \and
		Zhidong Bai
	}
	
	
	\institute{Ningning Xia \at School of Statistics and Management, Shanghai Key Laboratory of Financial Information Technology, Shanghai University of Finance and Economics, 777 Guo Ding Road, Shanghai 200433, China.\\
		The research was supported by NSFC 11501348, Shanghai Pujiang Program 15PJ1402300, IRTSHUFE and the state key program in the major research plan of NSFC 91546202.\\
		\email{xia.ningning@mail.shufe.edu.cn}           
		\and
		Z. D. Bai \at KLASMOE and School of Mathematics and Statistics, Northeast Normal
		University, 5268 People's Road, 130024 Changchun, China.\\
		The research was partially supported by National Natural Science Foundation of China (Grant No. 11571067).\\
		Tel.:(86) 431-8509-8161 \\
		Fax: (86) 431-8568-4987\\
		\email{baizd@nenu.edu.cn}           
	}
	
	\date{Received: date / Accepted: date}

	\maketitle
	
	\begin{abstract}
		In this paper, we adopt the eigenvector empirical spectral distribution (VESD) to investigate the limiting behavior of eigenvectors of a large dimensional Wigner matrix $\mathbf{W}_n$.
		In particular, we derive the optimal bound for the rate of convergence of the expected VESD of $\fW_n$ to the semicircle law, which is of order $O(n^{-1/2})$ under the assumption of having finite 10th moment. We further show that the convergence rates in probability and almost surely of the VESD are $O(n^{-1/4})$ and $O(n^{-1/6})$, respectively, under finite 8th moment condition. Numerical studies demonstrate that the convergence rate does not depend on the choice of unit vector involved in the VESD function, and the best possible bound for the rate of convergence of the VESD is of order $O(n^{-1/2})$. 
	\\
			{\bf Keywords:} Wigner matrices,  Eigenvectors, Empirical spectral distribution, Semicircular law,  Convergence rate\\
	\end{abstract}

\section{Introduction and main result}\label{introduction}
Let $\mathbf{W}_n=\dfrac{1}{\sqrt{n}}(X_{ij})_{i,j=1}^n$ be an $n\times n$ Hermitian matrix with diagonal entries being independent and identically distributed (i.i.d.) real random variables and above the diagonal entries being i.i.d. complex random variables.
The empirical spectral distribution (ESD) of $\mathbf{W}_n$ is defined by
\[
F^{W_n}(x):=\dfrac{1}{n}\sum_{i=1}^n \textrm{I}(\lambda_i\le x),
\]
where $\indic()$ is the indicator function. 
It is well-known that under  certain moment conditions, with probability one,
$F^{W_n}(x)$ converges to the semicircle law $F(x)$ whose probability density is given by
\[
F'(x)=\dfrac{1}{2\pi}\sqrt{4-x^2}, \qquad x\in [-2,2].
\]
The semicircle law is in fact the limiting spectral distribution of a very general class of matrices, which are frequently referred to as Wigner matrices in the literature.

Large dimensional random matrix theory has been rapidly developed ever since the semicircle law was established for the first time. 
The two basic classes of matrices, the Wigner matrices and sample covariance matrices, have held a basic position in large dimensional random matrix theory.  
The investigation of the spectrum of these two classes of matrices have been extensively developed in the literature (see, for example, \cite{Bai2010b,Bai2005,Li2014}). 
An interesting fact is that Wigner matrices and sample covariance matrices has a connection been set up by the semicircle law. 
When the ratio of the dimension and sample size goes to zero, the ESD of the normalized sample covariance matrices 
also converges to the semicircle law.  
Meanwhile, the semicircle law plays an important role in noncommutative probability and free independence. 
In the view of the value of research on the Wigner matrices to large dimensional random matrix theory, in this paper we consider the asymptotic properties of eigenvectors of large dimensional Wigner matrices. 

The earlier work about properties of eigenvectors may refer to \cite{Anderson1963}. When the observation dimension is fixed and low,
the author showed the asymptotic distribution for sample covariance eigenvectors when the observations are  multivariate normal distributed.  
For high dimensional case, Johnstone (2001) \cite{MR1863961} first proposed the spiked model to test the existence of principal component.
Further Paul (2007) \cite{MR2399865} investigated the direction of eigenvectors corresponding to spiked eignevalues and showed that the ordinary principal component analysis is necessarily inconsistent when the dimension to sample size ratio is a constant. 
The recent work regarding high dimensional principal components can be found in \cite{JL2009,PA2009,HS2015,GLZ2016}. 
The sparse properties of eigenvectors are also discussed in \cite{Zou2006,Shen2008,Croux2013,Ma2013}.
\cite{MR1003705,MR1062064,MR2330979,xia2013convergence,Bai2012a,Bao2013} characterized the asymptotical Haar property of the eigenmatrix (matrix of eigenvectors) based on the VESD function (see the definition \eqref{vesd} below).
The delocalization and universality property of eigenvectors can be found in \cite{ErdHos2009,ErdHos2012a,Knowles2013a,MR2930379}.

%

This article is concerned with the eigenvectors of $\mathbf{W}_n$.
Let $\mathbf{U}_n\fLa_n\mathbf{U}_n^*$ denote the spectral decomposition of $\mathbf{W}_n$, where $\fLa_n=\textrm{diag}(\lambda_1,\cdots,\lambda_n)$ with $\lambda_1\geq\cdots\geq\lambda_n$ and eigenmatrix $\mathbf{U}_n=(u_{ij})$ is a unitary matrix consisting of the corresponding orthonormal eigenvectors of $\mathbf{W}_n$.
Let $\mathbf{x}_n=(x_1,\cdots,x_n)^{\prime}\in\mathbb{C}^n$, $\|x_n\|=1$, be an arbitrary unit vector and $\mathbf{y}_n=(y_1,\cdots,y_n)^{\prime}=\mathbf{U}_n^*x_n$.
Motivated by the fact that the eigenmatrix $\fU_n$ of Gaussian real or complex Wigner matrices has Haar distribution if the signs of the first row of $\fU_n$ are i.i.d. symmetrically distributed, it is conjectured that the eigenmatrix $\fU_n$ is in a sense close to Haar distribution for general large Wigner matrices.
By adopting the idea of Silverstein, as demonstrated in \cite{MR1003705,MR1062064}, we define a stochastic process $Q_n(t)$ based on $\fU_n$  by
\[
Q_n(t):=\sqrt{\dfrac{n}{2}}\sum_{i=1}^{[nt]}\left(|y_i|^2-\dfrac{1}{n}\right),
\]
where $[a]$ denotes the integer not larger than $a$.
If $\mathbf{U}_n$ has a Haar measure on the topological group of $n$-dimensional orthogonal or unitary matrices, then $\fy_n$ has a uniform distribution over the unit sphere and  $Q_n(t)$ converges weakly to a Brownian bridge  by using Donsker's theorem.
Therefore, the convergence of process $Q_n(t)$ to a Brownian bridge characterizes the closeness of $\fU_n$ being Haar distributed.
For convenience, let us consider a transformation from $Q_n(t)$ to $Q_n(F^{\fW_n}(x))$ and define the eigenvector empirical spectral distribution (VESD) by
\begin{eqnarray}\label{vesd}
H^{W_n}(x):=\sum_{i=1}^n|y_i|^2\textrm{I}(\lambda_i\le x).
\end{eqnarray}
Then it follows that
\begin{eqnarray}
Q_n(F^{W_n}(x))=\sqrt{\dfrac{n}{2}}\left(H^{W_n}(x)-F^{W_n}(x)\right).
\label{f1f}
\end{eqnarray}
The recent result \cite{Bai2012a} proved that $H^{\fW_n}$ and $F^{\fW_n}$ have the same limiting distribution, the semicircle law, under finite 4th moment condition.
It is believed that the limiting theory of the two processes $Q_n(t)$ and $Q_n(F^{\fW_n})$ are equivalent. Thus the study of $Q_n(t)$ is converted to that of difference between the ESD $F^{\fW_n}$ and the VESD $H^{\fW_n}$.

The convergence rate of the ESD $F^{\fW_n}$ has been widely investigated in many literatures \cite{MR1217559,Bai1997,Bai2011,Gotze2003,Gotze2005}, but nothing has been done for the rate of the VESD $H^{\fW_n}$.
This article investigate three types of convergence rate of the VESD function $H^{\fW_n}(x)$.  We believe that good convergence rate of $H^{W_n}$ may help us investigate the behavior of process (\ref{f1f}), as reported in \cite{XB2015}. Besides, the convergence rate of the VESD has its own interest and importance in applications.  
It is well known that the optimal bound for the rate of convergence of the expected ESD to the semicircle law is of order $O(n^{-1})$, as shown in \cite{GT2016}. The best possible bound in the Gaussian case for the rate of convergence in probability is $O(n^{-1}\sqrt{\log n})$ (see \cite{Gotze2005}).  
Even though $F^{\fW_n}$ and $H^{\fW_n}$ both converge to semicircle law, a substantial difference exists between them.
Bai and Yao \cite{Bai2005} proved that
\[
n\int g(x)d\(F^{\fW_n}(x)-F(x)\)
\]
converges to a Gaussian distribution for an analytic function $g(x)$.
While Bai and Pan \cite{Bai2012a} showed that the limit of
\[
\sqrt{n} \int g(x)d\(H^{\fW_n}(x)-F(x)\)
\]
is a Gaussian distribution. 
Compared with the above two expressions, it is conjectured that the best possible bound for the rate of convergence of the VESD to the semicircle law is of order $O(n^{-1/2})$. 
The simulation studies in section \ref{simulation} further convince us that the optimal bound for $\Delta^H$ as defined in \eqref{bound_EVESD} is $O(n^{-1/2})$. 
This is a substantial difference between the ESD and the VESD of large dimensional Wigner matrices.

The Kolmogorov distances between two distribution functions are defined as
\begin{eqnarray}\label{bound_EVESD}
\De^H:=\|\E H^{\fW_n}-F\|=\sup_x|\E H^{\fW_n}(x)-F(x)|,
\end{eqnarray}
and
\[
\De_p^H:=\|H^{\fW_n}-F\|=\sup_x|H^{\fW_n}(x)-F(x)|,
\]
where $F(x)$ denotes semicircle law.
Notice that $\De_p^H$ provides bounds for the rate of convergence in probability.

The main theorems are established as follows,
\begin{thm}\label{thm1}
	Assume that $\{X_{ij},i>j=1,2,\cdots,n\}$ are i.i.d. complex random variables with $\textrm{E}X_{12}=0$, $\textrm{E}|X_{12}|^2=1$ and $\textrm{E}|X_{12}|^{10}<\infty$, that $\{X_{ii},i=1,\cdots,n\}$ are i.i.d. real random variables with $\textrm{E}X_{11}=0$ and $\textrm{E}|X_{11}|^2=1$. Further assume that all these random variables are independent.  Then,
	\[
	\|\textrm{E}H^{W_n}-F\|=O(n^{-1/2}),
	\]
	where $F$ is the distribution function of the semicircular law.
\end{thm}

\begin{thm}\label{thm2}
	Under the assumptions in theorem \ref{thm1} except that we now only require $\textrm{E}|X_{12}|^8<\infty$, $i,j=1\cdots,n$, we have
	\[
	\|H^{W_n}-F\|=O_p(n^{-1/4}).
	\]
\end{thm}

\begin{thm}\label{thm3}
	Under the assumptions in theorem \ref{thm2},  we have
	\[
	\|H^{W_n}-F\|=O_{a.s.}(n^{-1/6}).
	\]
\end{thm}

\begin{rem}
	We use notation $\xi_n=O_p(a_n)$ and $\eta_n=o_{a.s.}(b_n)$ if, for any $\ep>0$, there exist a large constant $C_1>0$ and a random variable $c_2>0$, such that
	\[
	P\(|\xi_n|/a_n\geq C_1\)\le \ep  ~~~~{\rm and }~~~~ P\(|\eta_n|/b_n\le c_2\)=1,
	\]
	respectively.
\end{rem}

\begin{rem}
	As explained before, we believe that theorem \ref{thm1} establishes the optimal bound for $\De^H$, which has a slower convergence rate than the expected ESD. The theoretical reason is as follows. We can see that from the proof of theorem \ref{thm1}, the bounds for both $\De^H$ and the expected ESD are derived in terms of their Stieltjes transforms (as defined in \eqref{ST}). The Stieltjes transform of the ESD only involves  $n$ equally-weighted diagonal elements of the resolvent matrix $(\fW_n-z\indic)^{-1}$. While the Stieltjes transform of the VESD considers both diagonal and off-diagonal elements with unequal weight. 
	The number of parameters increases from $n$ to $n^2$, and the weight is changed from equality to inequality.  
	Therefore,  it is not surprising that the VESD function $H^{\fW_n}$ has slower optimal convergence rate than the ESD $F^{\fW_n}$. 
\end{rem}

The rest of the paper is organized as follows.
Simulation studies are presented in Section \ref{simulation}.
Proof strategy is represented in Section \ref{strategy}.
Proofs of Theorem \ref{thm1}, \ref{thm2} and \ref{thm3} are given in Section \ref{proof1} and \ref{proof23}.
Several repeatedly employed results and existing lemmas are contained in Section \ref{Appendix1} and \ref{Appendix2}.
Truncation, centralization and rescaling are postponed in last section.

{\em Notation.} For Hermitian matrix $\fA$, $\|\fA\|=\la_{\max}(\fA)$ denotes its spectral norm.
For any vector $\fx$, $\|\fx\|$ denotes the Euclidean norm.
$\fX^*$ denotes the conjugate transpose of a matrix or vector $\fX$.
For any $z\in\bC$, we write $\Re(z)$, $\Im(z)$ as its real and imaginary parts, and $\bar{z}$ as its complex conjugate.
We also denote $\bC^+:=\{z\in\bC:\Re(z)>0\}$ and
$\cD:=\{z\in\bC: \Re(z)>0, \Im(z)\geq C_0n^{-1/2},|z|\le C\}$ for some positive constants $C_0$ and $C$.
Throughout this paper, $C$ denotes a generic constant, which may stand for different values at different appearances.
Let $\textrm{E}_k$ denote the conditional expectation given $\{X_{ij},k<i,j\le n\}$.


\section{Simulation Studies}\label{simulation}
In this section, we present some simulation studies to illustrate the behavior of the ESD and the VESD functions of Wigner matrix $\fW_n$ to the semicircle law.  
Further, under each setting, we plot Kolmogorov distances $\Delta_p^H$ as dimension $n$ increases to demonstrate the rate of convergence of the VESD function to the semicircle law.

In the simulation studies below, the Wigner matrix $\fW_n$ is taken to be $\frac{1}{\sqrt{2n}}(\fM+\fM^*)$, where $\fM=(M_{ij})$ whose entries are drawn independently from the standard normal distribution. 
To generate the unit vector $\fx_n$, we assume that $\fx_n=\frac{\fz}{\|\fz\|}$, with $\fz=(z_1,\cdots,z_n)$ and $z_i$ are drawn independently from the continuous uniform distribution on the interval $(0,1)$. 

In Figure \ref{plot}, each plot shows three distribution functions, those of the ESD, VESD and semicircle law. The dimension $n$ is taken to be 50, 500 and 2000, respectively. 
We can see that both the ESD and the VESD functions of $\fW_n$ converge to the semicircle law as dimension increases. But the VESD function seems more volatile than the ESD function. 
Compared with the ESD $F^{\fW_n}$, the VESD $H^{\fW_n}$ has a large deviation when $n$ is small, and it still has some distance when $n$ is as large as 2000. 

Next, we show the convergence rate of the VESD function of $\fW_n$ based on different choice of unit vector $\fx_n$. 
Denote the unit vector as $\fx_n=\frac{\fz}{\|\fz\|}$, where $\fz=(z_1,\cdots,z_n)$.  
In Figure \ref{rate}, we use black, blue, orange and green dashed lines to represent the Kolmogorov distance $\De_p^H$ when $z_i$ are drawn independently from uniform (0,1), standard normal, Poisson(1) and Binomial(10,0.6) distributions, respectively. 
The dimension $n$ increases from 50 to 1000. 
To show the rate of convergence of the VESD function $H^{\fW_n}$, under each setting, we take an average over 1000 repetitions to report one point as a Kolmogorov distance $\De_p^H$. 
Then we let the dimension $n$ increase from 50 to 1000.  
We can see that the convergence rate does not depend on the choice of unit vector $\fx_n$.  
Finally, we use red solid line to represent function $f(x)=1.15x^{-1/2}$, which matches the four curves very well. 
This implies that the optimal rate of convergence of the VESD function to the semicircle law is of order $O(n^{-1/2})$.  

\begin{figure}[h!]\centering
	\includegraphics[width=.45\textwidth]{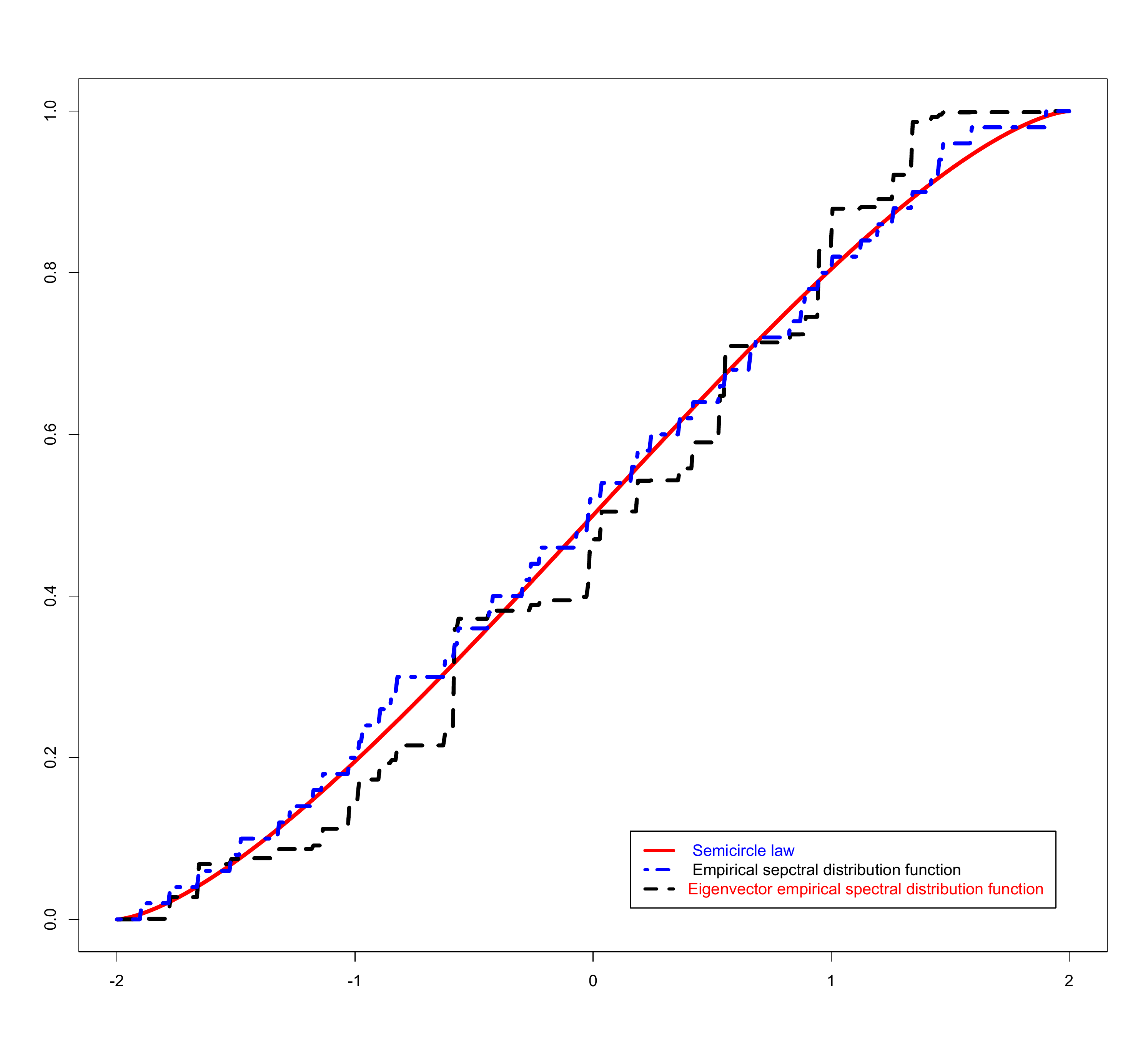}
	\includegraphics[width=.45\textwidth]{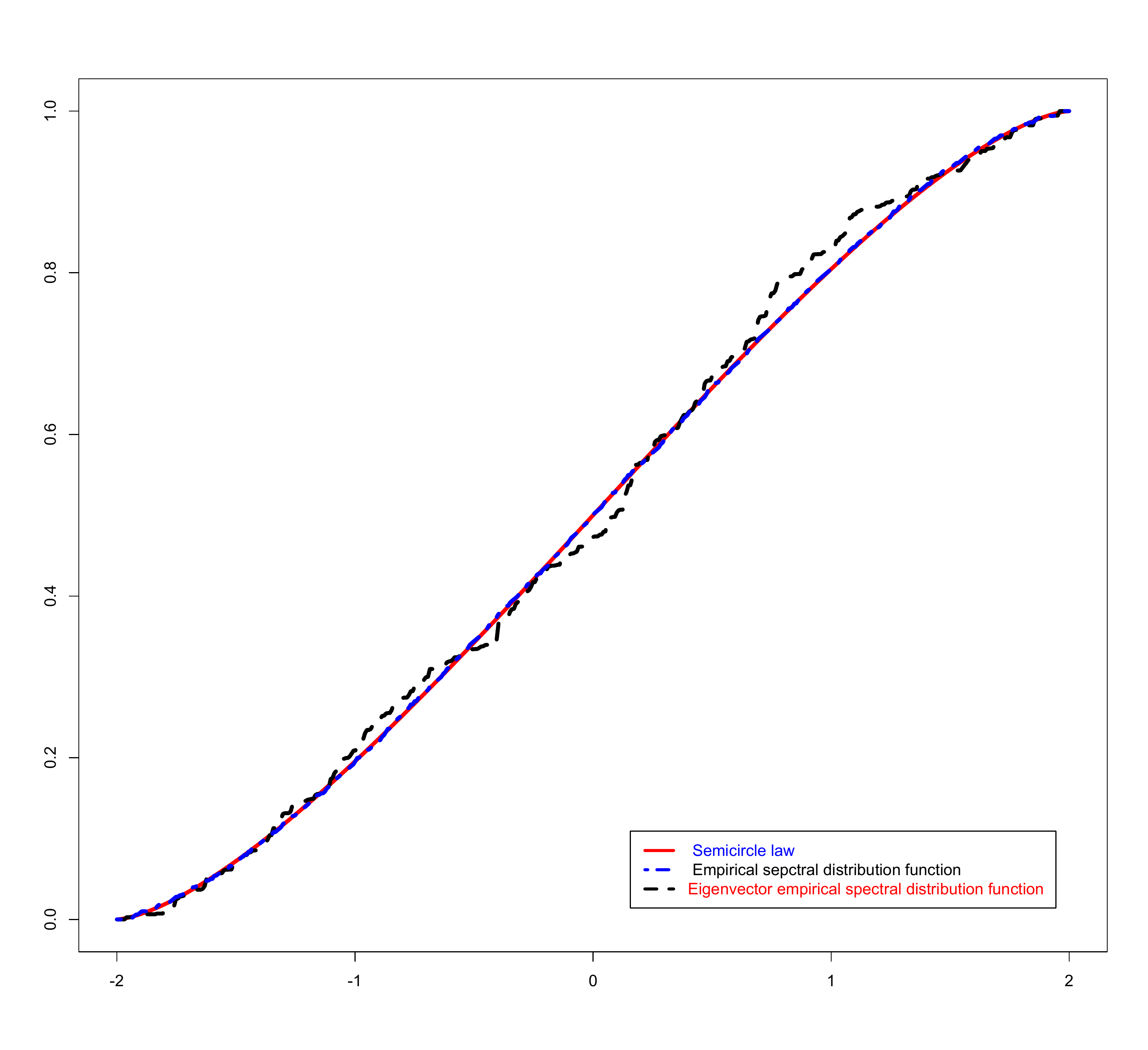}
	\includegraphics[width=.45\textwidth]{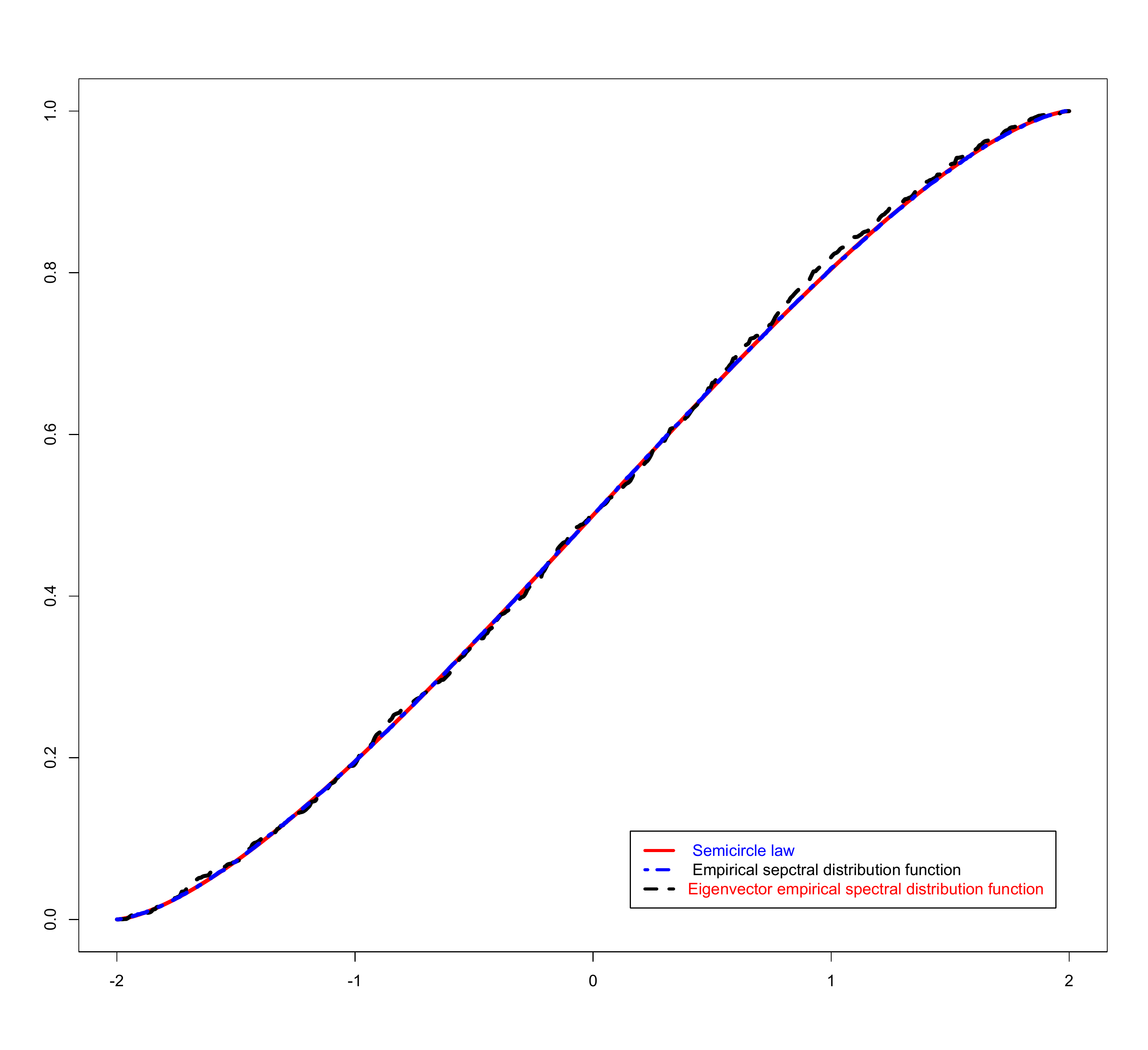}
	\caption{Each plot shows three functions of $\fW_n$, those of the empirical spectral distribution (ESD), the eigenvector empirical spectral distribution (VESD) and the semicircle law. The dimensions are $n=50,500$ and 2000, respectively.}
	\label{plot}
\end{figure}

\begin{figure}[h!]\centering
	\includegraphics[width=.5\textwidth]{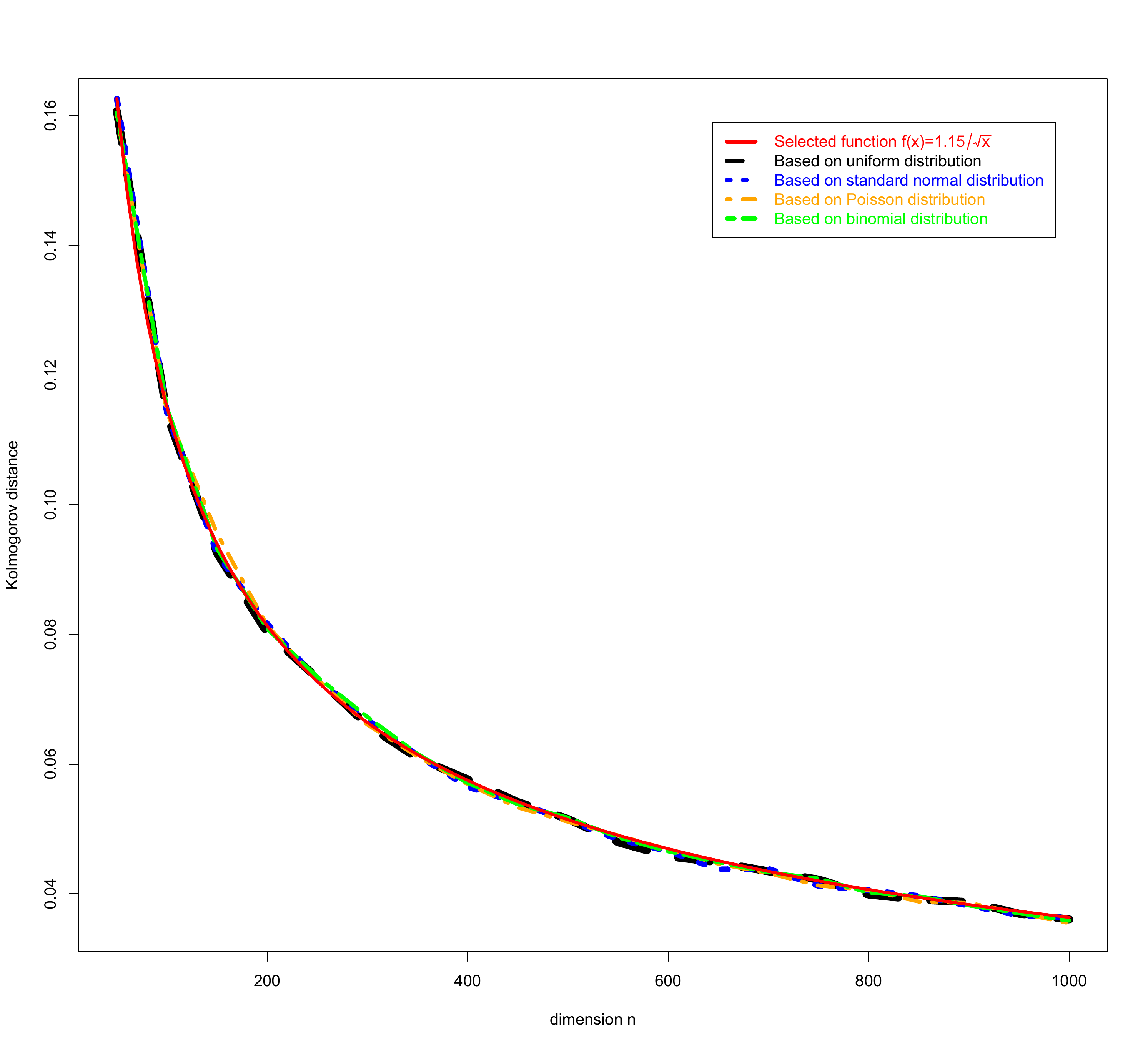}
	\caption{Denote the unit vector $\fx_n$ as $\fx_n=\fz/\|\fz\|$ with $\fz=(z_1,\cdots,z_n)$. As dimension $n$ increases from 50 to 1000, the black, blue, orange and green dashed lines represent the Kolmogorov distance $\De_p^H$ when $z_i$ are drawn independently from uniform(0,1), N(0,1), Poisson(1) and Binomial(10,0.6) distributions. Then we carefully choose the coefficient of $n^{-1/2}$ and use function $f(x)=1.15x^{-1/2}$to fit the four curves as shown in the red solid line.}
	\label{rate}
\end{figure}


\section{Proof Strategy}\label{strategy}
The required convergence rates of the VESD $H^{\fW_n}$ are established by using the following Berry-Esseen type inequality.
\begin{lem}(Theorem 2.2 in \cite{MR1217559})\label{Bai_lem}
	Let $F$ be a distribution function and let $G$ be a function of
	bounded variation satisfying $\int |F(x)-G(x)|dx < \infty$. Denote
	their Stieltjes transforms by $f(z)$ and $g(z)$, respectively. Then
	\begin{eqnarray}\label{Bai_inequality}
	&\quad&\|F-G\|= \sup_x |F(x)-G(x)| \nonumber\\
	&\leq & \dfrac{1}{\pi (1-\kappa)(2\gamma-1)}\bigg(\int_{-A}^{A} |f(z)-g(z)|du  +\frac{2\pi}{v}\int_{|x|>B} |F(x)-G(x)|dx \nonumber\\
	&& +\frac{1}{v}\sup_x \int_{|h|\leq 2v\tau} |G(x+h)-G(x)|dh\bigg),
	\end{eqnarray}
	where $z=u+iv$ is a complex variable, $\gamma$, $\kappa$, $\tau$, $A$ and $B$ are positive constants such that $A>B$,
	$$\kappa =\frac{4B}{\pi (A-B)(2\gamma-1)}<1,\quad\textrm{and}\quad\gamma=\frac{1}{\pi}\int_{|u|<\tau}\frac{1}{u^2+1}du>\frac{1}{2}.$$
\end{lem}
Here the Stieltjes transform for any distribution function $G(x)$ is defined by
\begin{eqnarray}\label{ST}
	s_G:=\int \dfrac{1}{x-z}dG(x), \qquad z\in \mathbb{C}^{+}=\{z\in\mathbb{C}, \Im z>0\}.
\end{eqnarray}
Then the Stieltjes transform of the ESD  $F^{\fW_n}(x)$ and the VESD $H^{\fW_n}(x)$ of Wigner matrices are denoted as
\[
s_n(z)=\dfrac{1}{n}\tr\(\fW_n-z\indic\)^{-1}, ~~~ {\rm and}~~~
s_n^H(z)=\fx_n^*\(\fW_n-z\indic\)^{-1}\fx_n,
\]
respectively, where $\mathbf{x}_n\in\mathbb{C}^n$, $\|\mathbf{x}_n\|=1$.

\begin{rem}
	In the proof of Theorem \ref{thm1}, inequality \eqref{Bai_inequality} is adopted to establish the convergence rate of $\E H^{\fW_n}$.
	We will see that notations $F$, $G$ in \eqref{Bai_inequality} will be replaced by
	the expected VESD $\E H^{\fW_n}$ and semicircle law $F$, respectively, and  the convergence rate of $\E H^{\fW_n}$ is mainly dominated by that of $\E s_n^H(z)$ and $v$ tending to 0.
	While constants $\ga,\kappa,\tau,A,B$ have nothing to do with the convergence rate.
	Moreover, in the proof, we can see that the convergence rate of $\E H^{\fW_n}$ also depends on that of $\E F^{\fW_n}$.
	As the optimal rate of $\E H^{\fW_n}$ is believed to be $O(n^{-1/2})$,
	therefore, the existing result about convergence rate of $\E F^{\fW_n}$
	in Lemma \ref{Delta1} would be enough to be used
	to determine the optimal convergence rate of $\E H^{\fW_n}$.
\end{rem}

\section{Proof of Theorem \ref{thm1}}\label{proof1}
Before proceeding with the proof of Theorem 1, it is necessary to truncate, re-centralize and re-normalize the random variables $X_{ij}$. We postpone the proof in Section \ref{truncation} and suppose that the underlying random variables satisfy $X_{ii}=0$ and for $i\neq j$, $|X_{ij}|\le \epsilon_nn^{1/4}$, $\textrm{E}X_{ij}=0$ and $\var(X_{ij})=1$, along with $\textrm{E}|X_{ij}|^8\le C$, where $\epsilon_n$ is a sequence decreasing to 0. The Wigner matrix $\fW_n=\dfrac{1}{\sqrt{n}}\fX_n$, where $\fX_n=(X_{ij})_{i,j=1}^n$.

Let $\mathbf{w}_k$ be the $k$th column of $\mathbf{W}_n$ and $\mathbf{e}_k$ be the $n\times 1$ column vector with the $k$th element being 1 and 0 otherwise. Then
\[
\fW_n=\sum_{k=1}^n \fw_k \fe_k^* ~~~~{\rm and}~~~~ \E\(\fw_k\fw_k^*\)=\dfrac{1}{n}\(\indic-\fe_k\fe_k^*\).
\]
Let $\mathbf{W}_{nk}$ be the matrix obtained from $\mathbf{W}_n$ by replacing all elements of its $k$th column and $k$th row with zero. Then
\[
\mathbf{W}_n=\mathbf{W}_{nk}+\mathbf{w}_k\mathbf{e}_k^*+\mathbf{e}_k\mathbf{w}_k^*.
\]

Denote
\[
\fA(z):=\fW_n-z\indic, ~~~~{\rm and}~~~~ \fA_k(z):=\fW_{nk}-z\indic.
\]
Then we get $\|\fA^{-1}(z)\|\le v^{-1}$, $\|\fA_k^{-1}(z)\|\le v^{-1}$ and
\begin{eqnarray}\label{AAk}
\mathbf{A}_k^{-1}(z)&=&\mathbf{A}^{-1}(z)+\mathbf{A}^{-1}(z)(\mathbf{A}(z)-\mathbf{A}_k(z))\mathbf{A}_k^{-1}(z)\nonumber\\
&=&\fA^{-1}(z)+\mathbf{A}^{-1}(z)(\mathbf{w}_k\mathbf{e}_k^*+\mathbf{e}_k\mathbf{w}_k^*)\mathbf{A}_k^{-1}(z).
\end{eqnarray}

Notice that all elements of the $k$th column and $k$th row of $\mathbf{A}_k^{-1}(z)$ are zero except that the $(k,k)$th element is $-1/z$. Consequently, we have
\begin{eqnarray}\label{Prop_Ak}
\mathbf{e}_k^*\mathbf{A}_k^{-1}(z)\mathbf{e}_k&=&-\dfrac{1}{z},~~~~~~~~~~ \mathbf{e}_k^*\mathbf{A}_k^{-1}(z)\mathbf{x}_n=-\dfrac{x_k}{z}, \nonumber\\ \mathbf{e}_k^*\mathbf{A}_k^{-1}(z)\mathbf{w}_k
&=&\fw_k^*\fA_k^{-1}(z)\fe_k=0.
\end{eqnarray}

For any $n\times n$ invertible matrice $\mathbf{B}$, vectors $\mathbf{a},\mathbf{b}\in\bC^n$ and scalar $q$, we have the following Sherman-Morrison formula
\begin{eqnarray}
(\mathbf{B}+q\mathbf{ab}^*)^{-1}&=&\mathbf{B}^{-1}-\dfrac{q\mathbf{B}^{-1}\mathbf{ab}^*\mathbf{B}^{-1}}{1+q\mathbf{b}^*\mathbf{B}^{-1}\mathbf{a}},\label{b1}\\
(\mathbf{B}+q\mathbf{ab}^*)^{-1}\mathbf{a}&=&\dfrac{\mathbf{B}^{-1}\mathbf{a}}{1+q\mathbf{b}^*\mathbf{B}^{-1}\mathbf{a}},\label{b1u}\\
\mathbf{b}^*(\mathbf{B}+q\mathbf{ab}^*)^{-1}&=&\dfrac{\mathbf{b}^*\mathbf{B}^{-1}}{1+q\mathbf{b}^*\mathbf{B}^{-1}\mathbf{a}},\label{vb1}
\end{eqnarray}

Further define
\[
\mathbf{W}_{k1}:=\mathbf{W}_{nk}+\mathbf{e}_k\mathbf{w}_k^*,~~~~~ \mathbf{W}_{k2}:=\mathbf{W}_{nk}+\mathbf{w}_k\mathbf{e}_k^*
\]
and for $j=1,2$,
\[
\fA_{kj}^{-1}(z):=\(\fW_{kj}-z\indic\)^{-1}.
\]

Denote
\begin{eqnarray*}
	\alpha_{k}&:=&\dfrac{1}{1+z^{-1}\mathbf{w}_k^*\mathbf{A}_k^{-1}(z)\mathbf{w}_k},\\
	\eta_k&:=&\mathbf{w}_k^*\mathbf{A}_k^{-1}(z)\mathbf{x}_n\mathbf{x}_n^*\mathbf{A}_k^{-1}(z)\mathbf{w}_k
	-\dfrac{1}{n}\left(\mathbf{x}_n^*\mathbf{A}_k^{-2}(z)\mathbf{x}_n-z^{-2}|x_k|^2\right),\\
	\gamma_k&:=&\dfrac{1}{1+\dfrac{1}{nz}\(\textrm{tr}\mathbf{A}_k^{-1}(z)+z^{-1}\)},\\ \xi_k&:=&\mathbf{w}_k^*\mathbf{A}_k^{-1}\mathbf{w}_k-\dfrac{1}{n}\(\tr\fA_k^{-1}+z^{-1}\).
\end{eqnarray*}
It is easy to see that
\[
\alpha_k=\gamma_k-z^{-1}\alpha_k\gamma_k\xi_k.
\]

First we will show the following useful facts:
\begin{eqnarray}
\fe_k^*\fA_{k1}^{-1}(z)\fw_k
&=&\fw_k^*\fA_{k2}^{-1}(z)\fe_k
=z^{-1}\fw_k^*\fA_{k}^{-1}(z)\fw_k \label{al_12}\\
\fA_{k1}^{-1}(z)&=&\fA_{k}^{-1}(z)-\fA_{k}^{-1}(z)\fe_k\fw_k^*\fA_{k}^{-1}(z)\label{Ak1}\\
\fA_{k2}^{-1}(z)&=&\fA_{k}^{-1}(z)-\fA_{k}^{-1}(z)\fw_k\fe_k^*\fA_{k}^{-1}(z)\label{Ak2}\\
\fA^{-1}(z)\fw_k&=&\al_k \fA_{k1}^{-1}(z)\fw_k \label{Awk}\\
\fA^{-1}(z)\fe_k&=&\al_k \fA_{k2}^{-1}(z)\fe_k \label{Aek}\\
\fe_k^*\fA^{-1}(z)&=&\al_k \fe_k^*\fA_{k1}^{-1}(z) \label{ekA}
\end{eqnarray}

By \eqref{Prop_Ak} and equation \eqref{b1}, it follows that
\begin{eqnarray*}
	\fe_k^*\fA_{k1}^{-1}(z)\fw_k&=&\fe_k^*\(\fA_{k}(z)+\fe_k\fw_k^*\)^{-1}\fw_k \\
	&=&\fe_k^*\(\fA_k^{-1}(z) -\dfrac{\fA_k^{-1}(z)\fe_k\fw_k^*\fA_k^{-1}(z)}{1+\fw_k^*\fA_k^{-1}(z)\fe_k}\)\fw_k\\
	&=&z^{-1}\fw_k^*\fA_k^{-1}(z)\fw_k.
\end{eqnarray*}
Similar to show that
$\fw_k^*\fA_{k2}^{-1}(z)\fe_k=z^{-1}\fw_k^*\fA_k^{-1}(z)\fw_k$.

\eqref{Ak1} and \eqref{Awk} follow from the facts that
\begin{eqnarray*}
	\mathbf{A}_{k1}^{-1}(z)&=&(\mathbf{W}_{nk}-z\mathbf{I}+\mathbf{e}_k\mathbf{w}_k^*)^{-1}
	=\mathbf{A}_k^{-1}(z)-\dfrac{\mathbf{A}_k^{-1}(z)\mathbf{e}_k\mathbf{w}_k^*\mathbf{A}_k^{-1}(z)}{1+\mathbf{w}_k^*\mathbf{A}_k^{-1}(z)\mathbf{e}_k}\\
	&=&\mathbf{A}_k^{-1}(z)-\mathbf{A}_k^{-1}(z)\mathbf{e}_k\mathbf{w}_k^*\mathbf{A}_k^{-1}(z).
\end{eqnarray*}
and by \eqref{b1u}
\begin{eqnarray*}
	\mathbf{A}^{-1}(z)\mathbf{w}_k
	=\(\mathbf{W}_{k1}-z\mathbf{I}+\mathbf{w}_k\mathbf{e}_k^*\)^{-1}\mathbf{w}_k
	=\dfrac{\mathbf{A}_{k1}^{-1}(z)\mathbf{w}_k}{1+\mathbf{e}_k^*\mathbf{A}_{k1}^{-1}(z)\mathbf{w}_k}
	=\alpha_k\mathbf{A}_{k1}^{-1}(z)\mathbf{w}_k,
\end{eqnarray*}
Moreover \eqref{Ak2}, \eqref{Aek} and \eqref{ekA} can be proved similarly.

Now, we write
\[
\mathbf{A}(z)=\fW_n-z\indic=\sum_{k=1}^n\mathbf{w}_k\mathbf{e}_k^*-z\mathbf{I},
\]
and multiply by $\mathbf{A}^{-1}(z)$ on both sides. It leads to
\[
z\mathbf{A}^{-1}(z)+\mathbf{I}
=\sum_{k=1}^n \mathbf{w}_k\mathbf{e}_k^*\mathbf{A}^{-1}(z).
\]
Multiplying $\mathbf{x}_n^*$ from left and $\mathbf{x}_n$ from right and then taking expectation, we have
\begin{eqnarray*}
	z\mathbf{x}_n^*\textrm{E}\mathbf{A}^{-1}(z)\mathbf{x}_n+1
	&=&\sum_{k=1}^n\textrm{E} \mathbf{x}_n^*\mathbf{w}_k\mathbf{e}_k^*\mathbf{A}^{-1}(z)\mathbf{x}_n \\
	&=&\sum_{k=1}^n\textrm{E}\al_k \mathbf{x}_n^*\mathbf{w}_k\mathbf{e}_k^*\mathbf{A}_{k1}^{-1}(z)\mathbf{x}_n
	~~~({\rm by}~ \eqref{ekA})\\
	&=&\sum_{k=1}^n\textrm{E}\al_k \mathbf{x}_n^*\mathbf{w}_k\mathbf{e}_k^*
	\(\mathbf{A}_k^{-1}(z)-\mathbf{A}_{k}^{-1}(z)\mathbf{e}_k\mathbf{w}_k^*\mathbf{A}_k^{-1}(z)\)\mathbf{x}_n
	~~~({\rm by}~ \eqref{Ak1})\\
	&=&L_1+L_2,
\end{eqnarray*}
where
\begin{eqnarray*}
	L_1&:=&\sum_{k=1}^n\textrm{E}\al_k \mathbf{x}_n^*\mathbf{w}_k\(\mathbf{e}_k^*\mathbf{A}_k^{-1}(z)\mathbf{x}_n\)
	=-\sum_{k=1}^n z^{-1}x_k\textrm{E}\al_k\mathbf{x}_n^*\mathbf{w}_k,\\
	L_2&:=&-\sum_{k=1}^n\textrm{E}\al_k \mathbf{x}_n^*\mathbf{w}_k\(\mathbf{e}_k^*\mathbf{A}_k^{-1}(z)\mathbf{e}_k\) \mathbf{w}_k^*\mathbf{A}_k^{-1}(z)\mathbf{x}_n
	=\sum_{k=1}^n z^{-1}\textrm{E}\al_k
	\mathbf{x}_n^*\mathbf{w}_k\mathbf{w}_k^*\mathbf{A}_k^{-1}(z)\mathbf{x}_n.
\end{eqnarray*}

Let $s(z)$ denote the Stieltjes transform of semicircle law, then $|s(z)|\le 1$ by (3.3) in \cite{MR1217559}.
Denote $s_n(z)$, $s_n^H(z)$ the Stieltjes transform of the ESD and the VESD, respectively.
Write $z:=u+iv$ with
\[
z\in\cD :=\{z:=u+iv: u>0, v\geq C_0n^{-1/2}, |z|\le C\} \subset \bC^+
\]
for some positive constant $C_0$ and $C$. Denote
\[
\Delta:=\left\|\textrm{E}F^{W_n}-F\right\|, ~~~~
\Delta^H:=\left\|\textrm{E}H^{W_n}-F\right\|
\]
the convergence rates of $\E F^{\fW_n}$ and $\E H^{\fW_n}$.
According to the definition of Stieltjes transform, integration by parts yields
\[
\lj\E s_n^H(z)-s(z)\rj\le \pi\De^H/v ~~~~{\rm and}~~~~
\lj\E s_n(z)-s(z)\rj\le \pi\De/v.
\]
Along with Lemma \ref{mem}, we have
\begin{eqnarray}\label{EsnH_bdd}
\E|s_n^H(z)|\le C(1+\De^H/v).
\end{eqnarray}

From (3.2) in \cite{MR1217559}, it is easy to verify that $s(z)$ satisfies the following equation,
\begin{eqnarray}
s(z)=-\dfrac{1}{z+s(z)}, \qquad \textrm{or} \qquad s^2(z)+zs(z)+1=0.
\label{sz}
\end{eqnarray}
Thus, we have
\begin{eqnarray*}
	|\textrm{E}s_n^H(z)-s(z)|&=&\left|\textrm{E}\mathbf{x}_n^*\mathbf{A}^{-1}(z)\mathbf{x}_n-s(z)\right|
	=\left|\textrm{E}\mathbf{x}_n^*\mathbf{A}^{-1}(z)\mathbf{x}_n+\dfrac{1}{z+s(z)}\right|\\
	&=&\dfrac{1}{|z+s(z)|}\left|(z+s(z))\textrm{E}\mathbf{x}_n^*\mathbf{A}^{-1}(z)\mathbf{x}_n+1\right|\\
	&=&\dfrac{1}{|z+s(z)|}\left|L_1+L_2+s(z)\textrm{E}\mathbf{x}_n^*\mathbf{A}^{-1}(z)\mathbf{x}_n\right|.
\end{eqnarray*}

By noticing that $\fw_k\fw_k^*$ and $\fA_k^{-1}(z)$ are independent and $\E(\fw_k\fw_k^*)=\dfrac{1}{n}\(\indic-\fe_k\fe_k^*\)$, and using \eqref{AAk}, we further write $L_2$ as
\begin{eqnarray*}
	L_2&=&\sum_{k=1}^nz^{-1}\E\al_k\fx_n^*\fw_k\fw_k^*\fA_k^{-1}(z)\fx_n\\
	&=&\sum_{k=1}^n z^{-1}\E\(\al_k+zs(z)\) \mathbf{x}_n^*\mathbf{w}_k\mathbf{w}_k^*\mathbf{A}_k^{-1}(z)\mathbf{x}_n
	-\sum_{k=1}^ns(z)\E\mathbf{x}_n^*\mathbf{w}_k\mathbf{w}_k^*\mathbf{A}_k^{-1}(z)\mathbf{x}_n\\
	&=&\sum_{k=1}^nz^{-1} \E(\alpha_k+zs(z)) \mathbf{x}_n^*\mathbf{w}_k\mathbf{w}_k^*\mathbf{A}_k^{-1}(z)\mathbf{x}_n
	+\dfrac{1}{n}\sum_{k=1}^n s(z) \textrm{E}\mathbf{x}_n^*\mathbf{e}_k\mathbf{e}_k^*\mathbf{A}_k^{-1}(z)\mathbf{x}_n\\
	&&-\dfrac{1}{n}\sum_{k=1}^ns(z)\textrm{E}\mathbf{x}_n^*\mathbf{A}_k^{-1}(z)\mathbf{x}_n
	\\
	&=&\sum_{k=1}^n z^{-1}\textrm{E}(\alpha_k+zs(z)) \mathbf{x}_n^*\mathbf{w}_k\mathbf{w}_k^*\mathbf{A}_k^{-1}(z)\mathbf{x}_n
	-\dfrac{1}{n}z^{-1}s(z)\sum_{k=1}^n|x_k|^2 \\
	&&-\dfrac{1}{n}\sum_{k=1}^ns(z)\textrm{E}\mathbf{x}_n^*\mathbf{A}^{-1}(z)\mathbf{x}_n -\dfrac{1}{n}\sum_{k=1}^ns(z) \textrm{E}\mathbf{x}_n^*\mathbf{A}^{-1}(z)(\mathbf{w}_k\mathbf{e}_k^*+\mathbf{e}_k\mathbf{w}_k^*)\mathbf{A}_k^{-1}(z)\mathbf{x}_n
	\\
	&=&\sum_{k=1}^n z^{-1}\textrm{E}(\alpha_k+zs(z)) \mathbf{x}_n^*\mathbf{w}_k\mathbf{w}_k^*\mathbf{A}_k^{-1}(z)\mathbf{x}_n
	-\dfrac{1}{nz}s(z)\\
	&&-s(z)\textrm{E}\mathbf{x}_n^*\mathbf{A}^{-1}(z)\mathbf{x}_n
	-\dfrac{1}{n}s(z)\sum_{k=1}^n\textrm{E}\mathbf{x}_n^*\mathbf{A}^{-1}(z)(\mathbf{w}_k\mathbf{e}_k^*+\mathbf{e}_k\mathbf{w}_k^*)\mathbf{A}_k^{-1}(z)\mathbf{x}_n.
\end{eqnarray*}
Therefore,
\[
L_2+s(z)\textrm{E}\mathbf{x}_n^*\mathbf{A}^{-1}(z)\mathbf{x}_n=L_{21}+L_{22}+L_{23},
\]
where
\begin{eqnarray*}
	L_{21}&:=&\sum_{k=1}^n z^{-1} \textrm{E}(\al_k+zs(z)) \mathbf{x}_n^*\mathbf{w}_k\mathbf{w}_k^*\mathbf{A}_k^{-1}(z)\mathbf{x}_n,\\
	L_{22}&:=&-\dfrac{s(z)}{n}\sum_{k=1}^n\textrm{E}\mathbf{x}_n^*\mathbf{A}^{-1}(z)(\mathbf{w}_k\mathbf{e}_k^*+\mathbf{e}_k\mathbf{w}_k^*)\mathbf{A}_k^{-1}(z)\mathbf{x}_n,\\
	L_{23}&:=&-\dfrac{1}{nz}s(z).
\end{eqnarray*}

By equation \eqref{sz} and fact $|s(z)|\le 1$, it follows that
\begin{eqnarray}\label{bdd_EST}
|\textrm{E}s_n^H(z)-s(z)|&=&\left|\textrm{E}\mathbf{x}_n^*\mathbf{A}^{-1}(z)\mathbf{x}_n-s(z)\right|\\
&=&\dfrac{1}{|z+s(z)|}\left|L_1+L_{21}+L_{22}+L_{23}\right|\\
&\le&|L_1|+|L_{21}|+|L_{22}|+|L_{23}|.
\end{eqnarray}
Obviously, $|L_{23}|\le \dfrac{C}{nv}$.

Since $\alpha_k=\gamma_k-z^{-1}\alpha_k\gamma_k\xi_k$, $\mathbf{w}_k$ and $\gamma_k$ are independent and $\textrm{E}\mathbf{w}_k=0$, we obtain
\begin{eqnarray*}
	|L_1|&=&\lj-\sum_{k=1}^nz^{-1}x_k\E\(\fx_n^*\fw_k\al_k\)\rj\\
	&=&\lj\sum_{k=1}^nz^{-2}x_k\E(\fx_n^*\fw_k\al_k\ga_k\xi_k)\rj\\
	&=&\lj\sum_{k=1}^nz^{-2}x_k\E(\fx_n^*\fw_k\ga_k^2\xi_k)
	-\sum_{k=1}^nz^{-3}x_k\E(\fx_n^*\fw_k\al_k\ga_k^2\xi_k^2)\rj\\
	&\le&|L_{11}|+|L_{12}|,
\end{eqnarray*}
where
\begin{eqnarray*}
	L_{11}:=\sum_{k=1}^nz^{-2}x_k\E(\fx_n^*\fw_k\ga_k^2\xi_k) ~~~{\rm and}~~~
	L_{12}:=\sum_{k=1}^nz^{-3}x_k\E(\fx_n^*\fw_k\al_k\ga_k^2\xi_k).
\end{eqnarray*}
Write $\fA_k^{-1}=(a_{ij})$, we then have
\begin{eqnarray*}
	\E\(\fx_n^*\fw_k\ga_k^2\xi_k\)
	&=&\E\(\fx_n^*\fw_k\ga_k^2\fw_k^*\fA_k^{-1}(z)\fw_k\)\\
	&=&\dfrac{1}{n\sqrt{n}}\E\ga_k^2\(\sum_{j=1}^n\bar{x}_jX_{jk}\) \(\sum_{j=1}^na_{jj}|X_{jk}|^2+\sum_{i\neq j}a_{ij}\bar{X}_{ik}X_{jk}\)\\
	&=&\dfrac{1}{n\sqrt{n}}\E\(\ga_k^2\sum_{j=1}^na_{jj}\bar{x}_j\) \E\(X_{jk}|X_{jk}|^2\).
\end{eqnarray*}
Notice that $\E|X_{12}|^8<\infty$ and $|z^{-1}\ga_k|$ is bounded by Lemma \ref{agb}, thus
\begin{eqnarray*}
	|L_{11}|\le\dfrac{C}{n\sqrt{n}} \(\sum_{k=1}^n|x_k|\)\(\sum_{j=1}^n|\bar{x}_j|\E|a_{jj}|\)
	\le \dfrac{C}{\sqrt{n}}(1+\De^H/v),
\end{eqnarray*}
where the last step comes from the facts that
\[
\dfrac{1}{n}\sum_{k=1}^n|x_k|\le \(\dfrac{1}{n}\sum_{k=1}^n|x_k|^2\)^{1/2}=\dfrac{1}{\sqrt{n}}
\]
and by \eqref{EsnH_bdd}
\[
\E|a_{jj}|=\E|\fe_j^*\fA_k^{-1}(z)\fe_j|
\le\max_{\|\fx_n\|=1}\E|\fx_n^*\fA_k^{-1}(z)\fx_n|
\le C(1+\De^H/v).
\]

By using Cauchy-Schwarz inequality, Lemma \ref{xi} and \ref{agb}, it follows that
\begin{eqnarray*}
	|L_{12}|&=&\left|\sum_{k=1}^nz^{-3}x_k\textrm{E}\mathbf{x}_n^*\mathbf{w}_k\alpha_k\gamma_k^2\xi_k^2\right|\\
	&\le&\sum_{k=1}^n|x_k|\left(\textrm{E}|\mathbf{x}_n^*\mathbf{w}_k|^2\textrm{E}|\xi_k|^4\right)^{1/2}
	\le\dfrac{C}{nv}(1+\Delta/v).
\end{eqnarray*}

Therefore, from the upper bounds of $|L_{11}|$ and $|L_{12}|$, we conclude that for $z\in\cD$,
\begin{eqnarray}\label{L1}
|L_1|\le\dfrac{C}{\sqrt{n}}(1+\Delta^H/v).
\end{eqnarray}

As in dealing with $L_{21}$,
by using equality $\al_k=\ga_k-z^{-1}\al_k\ga_k\xi_k$ three times, we further decompose term  $L_{21}$ as follows,
\begin{eqnarray*}
	L_{21}&=&\sum_{k=1}^n\textrm{E}\mathbf{x}_n^*\mathbf{w}_k\mathbf{w}_k^*\mathbf{A}_k^{-1}(z)\mathbf{x}_n(\alpha_k+zs(z))z^{-1}\\
	&=&\sum_{k=1}^n\textrm{E}\mathbf{x}_n^*\mathbf{w}_k\mathbf{w}_k^*\mathbf{A}_k^{-1}(z)\mathbf{x}_n(\gamma_k+zs(z))z^{-1}
	-\sum_{k=1}^n\textrm{E}\mathbf{x}_n^*\mathbf{w}_k\mathbf{w}_k^*\mathbf{A}_k^{-1}(z)\mathbf{x}_nz^{-2}\alpha_k\gamma_k\xi_k\\
	&=&\sum_{k=1}^n\textrm{E}\mathbf{w}_k^*\mathbf{A}_k^{-1}(z)\mathbf{x}_n\mathbf{x}_n^*\mathbf{w}_k(\gamma_k+zs(z))z^{-1}
	-\sum_{k=1}^n\textrm{E}\mathbf{w}_k^*\mathbf{A}_k^{-1}(z)\mathbf{x}_n\mathbf{x}_n^*\mathbf{w}_kz^{-2}\gamma_k^2\xi_k\\
	&&+\sum_{k=1}^n\textrm{E}\mathbf{w}_k^*\mathbf{A}_k^{-1}(z)\mathbf{x}_n\mathbf{x}_n^*\mathbf{w}_kz^{-3}\gamma_k^3\xi_k^2
	+\sum_{k=1}^n\textrm{E}\mathbf{w}_k^*\mathbf{A}_k^{-1}(z)\mathbf{x}_n\mathbf{x}_n^*\mathbf{w}_kz^{-4}\alpha_k\gamma_k^3\xi_k^3\\
	&:=&L_{211}+L_{212}+L_{213}+L_{214}.
\end{eqnarray*}

For the first term $L_{211}$, notice that $\fw_k$ is independent with $\fA_k^{-1}(z)$ and $\ga_k$, then by Cauchy-Schwarz inequality, \eqref{EsnH_bdd}, and Lemma \ref{asz}, we obtain
\begin{eqnarray*}
	|L_{211}|&=&\lj\sum_{k=1}^n\E\fw_k^*\fA_k^{-1}(z)\fx_n\fx_n^*\fw_k(\ga_k+zs(z))z^{-1}\rj\\
	&=&\left|n^{-1}\sum_{k=1}^n\textrm{E}\mathbf{x}_n^*\mathbf{A}_k^{-1}(z)\mathbf{x}_n(\gamma_k+zs(z))z^{-1}\right.\\
	&&~~~~~~~~~~~~~~~~~~~~~~~~~~~
	\left.+n^{-1}\sum_{k=1}^n\textrm{E}\mathbf{e}_k^*\mathbf{A}_k^{-1}(z)\mathbf{x}_n\mathbf{x}_n^*\mathbf{e}_k(\gamma_k+zs(z))z^{-1}\right|\\
	&\le&n^{-1}\sum_{k=1}^n\left(\textrm{E}|\mathbf{x}_n^*\mathbf{A}_k^{-1}(z)\mathbf{x}_n|^2
	\times \E\lj\(\gamma_k+zs(z)\) z^{-1}\rj^2\right)^{1/2}\\
	&&~~~~~~~~~~~~~~~~~~~~~~~~~~~~~
	+\dfrac{1}{zn}\sum_{k=1}^n|x_k|^2\left(\textrm{E}
	\lj\(\gamma_k+zs(z)\)z^{-1}\rj^2\right)^{1/2}\\
	&\le&\dfrac{C}{nv}(1+\Delta^H/v)+\dfrac{C}{(nv)^2}\\
	&\le&\dfrac{C}{nv}(1+\Delta^H/v).
\end{eqnarray*}

For the second term $L_{212}$.
Observe that
\[
\textrm{E}\(\mathbf{x}_n^*\mathbf{A}_k^{-1}(z)\mathbf{x}_n-z^{-1}|x_k|^2\)z^{-2}\gamma_k^2\xi_k=0.
\]
In view of the definition $\xi_k$, we apply Lemma \ref{xabx} and \ref{agb} to obtain
\begin{eqnarray*}
	|L_{212}|&=&\lj\sum_{k=1}^n\E \(\fw_k^*\fA_k^{-1}(z)\fx_n\fx_n^*\fw_k -\dfrac{1}{n}\(\fx_n^*\fA_k^{-1}(z)\fx_n-z^{-1}|x_k|^2\)\) \right.\\
	&&~~~\left.\times z^{-2}\ga_k^2\(\fw_k^*\fA_k^{-1}(z)\fw_k -\dfrac{1}{n}\(\tr\fA_k^{-1}(z)+z^{-1}\)\) \rj \\
	&=&\bigg|\dfrac{1}{n^2}\sum_{k=1}^n \bigg[ \(\E|X_{12}|^4-|\E X_{12}|^2-2\) \E\sum_{i=1}^n (\fA_k^{-1}(z)\fx_n\fx_n^*)_{ii}(\fA_k^{-1}(z))_{ii}z^{-2}\ga_k^2 \\
	&&~~ +|\E X_{12}^2|^2 \E\tr\(\fA_k^{-1}(z)\fx_n\fx_n^*\fA_k^{-1}(\bar{z})\)z^{-2}\ga_k^2 \\
	&&~~~ +\E\tr\(\fA_k^{-1}(z)\fx_n\fx_n^*\fA_k^{-1}(z)\)z^{-2}\ga_k^2 ~~ \bigg]\bigg|\\
	&\le&\dfrac{C}{n^2}\sum_{k=1}^n \(\E\sum_{i=1}^n|\fe_i^*\fA_k^{-1}(z)\fx_n|^2\)^{1/2} \(\sum_{i=1}^n|\fx_n^*\fe_i|^2\E|(\fA_k^{-1}(z))_{ii}|^2\)^{1/2} \\ &&~~~+\dfrac{C}{nv}(1+\De^H/v)\\
	&\le& \dfrac{C}{nv}(1+\De^H/v).
\end{eqnarray*}

For the third term $L_{213}$, first decompose $L_{213}$ as
\[
L_{213}=L_{213}^{(1)}+L_{213}^{(2)}+L_{213}^{(3)},
\]
where
\begin{eqnarray*}
	L_{213}^{(1)}&:=&\sum_{k=1}^n\E (z^{-1}\ga_k)^3 \xi_k^2
	\(\fw_k^*\fA_k^{-1}(z)\fx_n\fx_n^*\fw_k-\dfrac{1}{n}\(\fx_n^*\fA_k^{-1}(z)\fx_n-z^{-1}|x_k|^2\)\);\\
	L_{213}^{(2)}&:=&\dfrac{1}{n}\sum_{k=1}^n\E(z^{-1}\ga_k)^3\xi_k^2(\fx_n^*\fA_k^{-1}(z)\fx_n);\\
	L_{213}^{(3)}&:=&-\dfrac{z^{-1}}{n}\sum_{k=1}^n|x_k|^2\E(z^{-1}\ga_k)^3\xi_k^2.
\end{eqnarray*}
Using the same proof strategy as shown (4.5) in \cite{xia2013convergence}, one can similarly prove that
\[
|L_{213}^{(1)}|\le \dfrac{C}{nv}(1+\De^H/v).
\]
By Cauchy-Schwarz inequality, and Lemma \ref{agb}, \ref{xi}, it leads to
\begin{eqnarray*}
	|L_{213}^{(2)}|&\le&\dfrac{C}{n}\sum_{k=1}^n\(\E|\xi_k|^4\)^{1/2}
	\(\E|\fx_n^*\fA_k^{-1}(z)\fx_n|^2\)^{1/2}\\
	&\le&\dfrac{C}{nv}(1+\De/v)(1+\De^H/v).
\end{eqnarray*}
And
\[
|L_{213}^{(3)}|~\le~\dfrac{C}{n|z|}\sum_{k=1}^n|x_k|^2\E|\xi_k^2|
~\le~\dfrac{C}{(nv)^2}(1+\De/v).
\]
Thus for $z\in\cD$, we conclude that
\begin{eqnarray*}
	|L_{213}|\le\dfrac{C}{nv}(1+\De^H/v).
\end{eqnarray*}

For the last term $L_{214}$, by Cauchy-Schwarz inequality and Lemma \ref{xtrx},\ref{xi},
it follows that
\begin{eqnarray*}
	|L_{214}|&\le& C\sum_{k=1}^n \bigg[
	\(\E|\fw_k^*\fA_k^{-1}(z)\fx_n\fx_n^*\fw_k-\dfrac{1}{n}\fx_n^*\fA_k^{-1}(z)\fx_n|^2
	+\dfrac{1}{n^2}\E|\fx_n^*\fA_k^{-1}(z)\fx_n|^2\)^{1/2} \\
	&&~~~~~~~~~~~~~ \times \(\E|\xi_k|^6\)^{1/2}\bigg]\\
	&\le&\dfrac{C}{n^{3/2}v^2}(1+\De^H/v)^{1/2}(1+\De/v)^{3/2}.
\end{eqnarray*}
Therefore, along with all the upper bounds of $L_{211},L_{212},L_{213},L_{214}$,
we obtain that for $z\in\cD$,
\begin{eqnarray}\label{L21}
|L_{21}|\le\dfrac{C}{nv}(1+\De^H/v).
\end{eqnarray}

We now turn to term $L_{22}$.
\begin{eqnarray*}
	L_{22}&=&-\dfrac{s(z)}{n}\sum_{k=1}^n\textrm{E}\mathbf{x}_n^*\mathbf{A}^{-1}(z)\mathbf{w}_k \mathbf{e}_k^*\mathbf{A}_k^{-1}(z)\mathbf{x}_n
	-\dfrac{s(z)}{n}\sum_{k=1}^n\textrm{E}\mathbf{x}_n^*\mathbf{A}^{-1}(z)\mathbf{e}_k \mathbf{w}_k^*\mathbf{A}_k^{-1}(z)\mathbf{x}_n\\
	&:=&L_{221}+L_{222}.
\end{eqnarray*}
Appealing to \eqref{Prop_Ak}, \eqref{Awk} and \eqref{Ak1}, one can see that
\begin{eqnarray*}
	L_{221}&=&-\dfrac{s(z)}{n}\sum_{k=1}^n\textrm{E}\mathbf{x}_n^*\mathbf{A}^{-1}(z)\mathbf{w}_k \(\mathbf{e}_k^*\mathbf{A}_k^{-1}(z)\mathbf{x}_n\)\\
	&=&\dfrac{s(z)}{zn}\sum_{k=1}^nx_k\textrm{E}\mathbf{x}_n^*\mathbf{A}^{-1}(z)\mathbf{w}_k\\
	&=&\dfrac{s(z)}{zn}\sum_{k=1}^nx_k\textrm{E}\mathbf{x}_n^*\mathbf{A}_{k1}^{-1}(z)\mathbf{w}_k\alpha_k\\
	&=&\dfrac{s(z)}{zn}\sum_{k=1}^nx_k\left(\textrm{E}\mathbf{x}_n^*\mathbf{A}_k^{-1}(z)\mathbf{w}_k\alpha_k
	-\textrm{E}\mathbf{x}_n^*\mathbf{A}_k^{-1}(z)\mathbf{e}_k\mathbf{w}_k^*\mathbf{A}_k^{-1}(z)\mathbf{w}_k\alpha_k\right)\\
	&=&\dfrac{s(z)}{zn}\sum_{k=1}^nx_k\textrm{E}\mathbf{x}_n^*\mathbf{A}_k^{-1}(z)\mathbf{w}_k\alpha_k
	-\dfrac{s(z)}{z^2n}\sum_{k=1}^n|x_k|^2\textrm{E}\mathbf{w}_k^*\mathbf{A}_k^{-1}(z)\mathbf{w}_k\alpha_k.
\end{eqnarray*}
For $|z^{-1}\alpha_k|\le C$ by Lemma \ref{agb}, Cauchy-Schwarz inequality leads us to
\begin{eqnarray*}
	&&\left|\dfrac{s(z)}{zn}\sum_{k=1}^nx_k\textrm{E}\mathbf{x}_n^*\mathbf{A}_k^{-1}(z)\mathbf{w}_k\alpha_k\right|
	\le \dfrac{C}{n}\sum_{k=1}^n|x_k|\left(\textrm{E}
	\lj\mathbf{x}_n^*\mathbf{A}_k^{-1}(z)\mathbf{w}_k\rj^2\right)^{1/2}\\
	&=&\dfrac{C}{n}\sum_{k=1}^n|x_k|\(\E\(\fw_k^*\fA_k^{-1}(\bar{z})\fx_n \fx_n^*\fA_k^{-1}(z)\fw_k\)\)^{1/2}\\
	&=&\dfrac{C}{n}\sum_{k=1}^n|x_k| \(\dfrac{1}{n}\E \(\fx_n^*\fA_k^{-1}(z)\fA_k^{-1}(\bar{z})\fx_n\)-\dfrac{|x_k|^2}{n|z|^2}\)^{1/2}\\
	&=&\dfrac{C}{n}\sum_{k=1}^n|x_k| \(\dfrac{1}{nv}\Im\E \(\fx_n^*\fA_k^{-1}(z)\fx_n\)-\dfrac{|x_k|^2}{n|z|^2}\)^{1/2}\\
	&\le&\dfrac{C}{n\sqrt{v}}(1+\Delta^H/v)^{1/2},
\end{eqnarray*}
and
\begin{eqnarray*}
	&&\lj\dfrac{s(z)}{z^2n}\sum_{k=1}^n|x_k|^2\E \fw_k^*\fA_k^{-1}(z)\fw_k\al_k \rj\\
	&\le&\dfrac{C}{nv}\sum_{k=1}^n|x_k|^2\(\E\lj\fw_k^*\fA_k^{-1}(z)\fw_k\rj^2\)^{1/2}\\
	&\le&\dfrac{C}{nv}\sum_{k=1}^n|x_k|^2
	\(\E\lj\fw_k^*\fA_k^{-1}(z)\fw_k-n^{-1}\tr\fA_k^{-1}\rj^2
	+\E\lj n^{-1}\tr\fA_k^{-1}(z)\rj^2\)^{1/2}\\
	&\le&\dfrac{C}{nv},
\end{eqnarray*}
where the last inequality is due to Lemma \ref{xtrx}, \ref{6.7}, \ref{Delta1}.

Thus
\[
|L_{221}|\le\dfrac{C}{n\sqrt{v}}(1+\De^H/v)^{1/2}+\dfrac{C}{nv}.
\]
Similarly one can prove that $|L_{222}|$ has the same upper bound as $|L_{221}|$.
Thus we obtain
\begin{eqnarray}\label{L22}
|L_{22}|\le\dfrac{C}{n\sqrt{v}}(1+\De^H/v)^{1/2}+\dfrac{C}{nv}.
\end{eqnarray}

As it has been shown in \eqref{bdd_EST}, \eqref{L1}, \eqref{L21} and \eqref{L22}, we conclude
that for $z\in\cD$,
\begin{eqnarray*}\label{EST}
	|\E s_n^H(z)-s(z)|&\le&|L_1|+|L_{21}|+|L_{22}|+|L_{23}|\\
	&\le&\dfrac{C}{nv}(1+\De^H/v).
\end{eqnarray*}

Note that
\[
\lj F(x+h)-F(x)\rj=\lj\dfrac{1}{2\pi}\int_x^{x+h}\sqrt{4-t^2}dt\rj\le\dfrac{1}{\pi}|h|.
\]
It is easy to verify that
\[
v^{-1}\sup_x\int_{|h|\le 2v\tau} |F(x+h)-F(x)|dh \le Cv,
\]
where $\tau$ is a constant.

Since the support of semicircle law is $[-2,2]$, by choosing $B=3$ in
inequality \eqref{Bai_inequality}  and from Lemma \ref{tail}, we can show that for any $t>0$,
\[
\int_{|x|>3}\lj\E H^{\fW_n}(x)-F(x)\rj dx=o(n^{-t}).
\]
Along with Lemma \ref{Bai_lem} and choose $A=16$ in inequality \eqref{Bai_inequality},
then we have
\begin{eqnarray*}
	\Delta^H&=&\|\textrm{E}H^{W_n}-F\|
	\leq  C\bigg(\int_{-16}^{16} |\textrm{E}s_n^H(z)-s(z)|du  \\
	&&+\frac{2\pi}{v}\int_{|x|>3} |\textrm{E}H^{W_n}(x)-F(x)|dx
	+\frac{1}{v}\sup_x \int_{|h|\leq 2v\tau} |G(x+h)-G(x)|dh\bigg)\\
	&\le& \dfrac{C}{nv}(1+\Delta^H/v)+Cv.
\end{eqnarray*}
For any $z\in\cD$. That is there exists a large constant $C_0$, such that
$v\geq C_0n^{-1/2}$ and $\dfrac{C}{nv^2}<\dfrac{2}{3}$. It follows that
\[
\De^H \le \dfrac{3C}{nv}+3Cv ~=~ O(n^{-1/2}).
\]
The proof of Theorem \ref{thm1} is complete.

\section{Proof of Theorem \ref{thm2} and Theorem \ref{thm3}}\label{proof23}

The convergence rate $\De^H=O(n^{-1/2})$ has already been established in Theorem \ref{thm1},
now we only need to focus on the upper bound of $\E|s_n^H(z)-\E s_n^H(z)|$.
By Lemma \ref{mem} and Cauchy-Schwarz inequality, it follows that
\[
\E\lj s_n^H(z)-\E s_n^H(z)\rj \le \dfrac{C}{\sqrt{n}v}(1+\De^H/v).
\]
Therefore, for $A=16$ and $z\in\cD^*:=\{z=u+iv:u>0,v\geq Cn^{-1/4}\}$, by Lemma \ref{Bai_lem},
we conclude that
\begin{eqnarray*}
	\E\De_p^H&:=&\E\|H^{\fW_n}(x)-F(x)\|\\
	&\le& \E\|H^{\fW_n}(x)-\E H^{\fW_n}(x)\|+\|\E H^{\fW_n}(x)-F(x)\|\\
	&\le&C\int_{-16}^{16}\E\lj s_n^H(z)-\E s_n^H(z) \rj du+\De^H\\
	&\le&\dfrac{C}{\sqrt{n}v}(1+\De^H/v)+\De^H
	~=~ O(n^{-1/4}).
\end{eqnarray*}
This completes the proof of Theorem \ref{thm2}.

The proof of Theorem \ref{thm3} can be proved almost the same as that of Theorem \ref{thm2}.
We need to show that
\begin{eqnarray}\label{thm3_bdd}
\int_{-16}^{16} \E\lj s_n^H(z)-\E s_n^H(z)\rj du=O_{a.s.}(n^{-1/6}).
\end{eqnarray}
By Lemma \ref{mem} and choosing $v=n^{-1/6}$, it follows that for $l\geq 8$,
\[
n^{2l/6}\E\lj s_n^H(z)-\E s_n^H(z)\rj^{2l}
\le \dfrac{Cn^{l/3}}{n^2(nv^2)^{l/2}}=\dfrac{C}{n^2}.
\]
Therefore, \eqref{thm3_bdd} follows by Borel-Cantelli lemma. The proof of Theorem \ref{thm3} is complete.

\section{Appendix I}\label{Appendix1}
\begin{lem}\label{xi}
	Under the conditions of Theorem \ref{thm1}, when $z\in\cD$, there exists a constant $C$, such that for any $l\geq 1$,
	\[
	\textrm{E}|\xi_k|^{2l}\le
	\begin{cases} \dfrac{C}{n^lv^l}\left(1+\dfrac{\Delta}{v}\right)^l, \quad  & \mbox{for} \quad l\le 3,\\
	\dfrac{C}{n^{l+1}v^{2l-1}}\left(1+\dfrac{\Delta}{v}\right),\quad & \mbox{for} \quad l\geq 4.
	\end{cases}
	\]
	where $\Delta=\|\E F^{\fW_n}-F\|$.
\end{lem}
\begin{proof}
	For Stieltjes transform $s_n(z)$, integration by parts yields
	\[|\E s_n(z)-s(z)|\le \pi\De/v.\]
	Combined with the fact $|s(z)|\le 1$, implies that  $|\E s_n(z)|\le C(1+\De/v)$.
	Under finite 8th moment assumption, we have $\E|X_{jk}|^{4l}\le Cn^{l-2}$. Thus by Lemma \ref{xtrx}, it follows that
	\begin{eqnarray*}
		\E|\xi_k|^{2l}&\le&\dfrac{C}{n^{2l}}\lb\(\E\tr\fA_k^{-1}(z)\fA_k^{-1}(\bar{z})\)^l
		+n^{l-2}\E\tr\(\fA_k^{-1}(z)\fA_k^{-1}(\bar{z})\)^l \rb \\
		&=&\dfrac{C}{n^{2l}}\lb \(v^{-1}\E\(\Im\tr\fA_k^{-1}(z)\)\)^l
		+n^{l-2}\E\tr\(\fA_k^{-1}(z)\fA_k^{-1}(\bar{z})\)^l \rb \\
		&\le&\dfrac{C}{n^{2l}}\lb \dfrac{n^l}{v^l}|\E\Im(s_n(z))|^l +\dfrac{n^{l-1}}{v^{2l-1}} \E|\Im(s_n(z))|\rb\\
		&\le&\dfrac{C}{n^lv^l}(1+\De/v)^l+\dfrac{C}{n^{l+1}v^{2l-1}}(1+\De/v).
	\end{eqnarray*}
\end{proof}

\begin{lem}\label{agb}
	Under the conditions of Theorem \ref{thm1}, for all $z\in\cD$, there exists a large constant $C>0$, such that for any $t>0$ and for all $k=1,\cdots,n$, we have
	\[
	P(|z^{-1}\ga_n|>C)=o(n^{-t}), ~~~~~ P(|z^{-1}\al_k|>C)=o(n^{-t}).
	\]
\end{lem}
\begin{proof}
	Recall that
	\[
	\alpha_k=\dfrac{1}{1+z^{-1}\mathbf{w}_k^*\mathbf{A}_k^{-1}\mathbf{w}_k}, \qquad \gamma_k=\dfrac{1}{1+z^{-1}n^{-1}(\textrm{tr}\mathbf{A}_k^{-1}+z^{-1})}.
	\]
	Define
	\[
	b_n:= \dfrac{1}{1+z^{-1}n^{-1}\textrm{Etr}\mathbf{A}^{-1}(z)}, ~~~~~
	\beta_n:=\dfrac{1}{1+z^{-1}n^{-1}\textrm{tr}\mathbf{A}^{-1}(z)}.
	\]
	First if we suppose that for any fixed $t>0$, there exists a large constant $C>0$, such that the following equation holds,
	\begin{eqnarray}\label{zbeta}
	P(|z^{-1}\beta_n|>C)=o(n^{-t}).
	\end{eqnarray}
	Then from
	\begin{eqnarray}\label{beta_ga}
	\beta_n-\gamma_k=\beta_n\gamma_kz^{-1}n^{-1}\(\tr\fA_k^{-1}(z)-\tr\fA^{-1}(z)+z^{-1}\),
	\end{eqnarray}
	we have
	\begin{eqnarray*}
		|z^{-1}\ga_k|&=&\dfrac{|z^{-1}\beta_n|}{|1+z^{-1}\beta_nn^{-1}\(\tr(\fA_k^{-1}(z)-\fA^{-1}(z))+z^{-1}\)|}\\
		&\le&\dfrac{|z^{-1}\beta_n|}{1-2|z^{-1}\beta_n|(nv)^{-1}},
	\end{eqnarray*}
	where the last inequality comes from the facts that
	$|\tr(\fA_k^{-1}(z)-\fA^{-1}(z))|\le v^{-1}$ and $|z^{-1}|\le v^{-1}$.
	If \eqref{zbeta} is true, for $v\geq O(n^{-1/2})$, we can choose $n$ large enough such that $2|z^{-1}\beta_n|(nv)^{-1}\le 1/2$.
	Then $|z^{-1}\ga_k|\le 2|z^{-1}\beta_n|$, which implies that $P(|z^{-1}\ga_k|>C)=o(n^{-t})$ holds.
	
	Similarly, consider $z^{-1}\al_k$, if $|z^{-1}\ga_k||\xi_k|\le 1/2$ holds, then we have
	\[
	|z^{-1}\alpha_k|=\dfrac{|z^{-1}\gamma_k|}{|1+z^{-1}\gamma_k\xi_k|}
	\le\dfrac{|z^{-1}\gamma_k|}{1-|z^{-1}\gamma_k||\xi_k|}
	\le 2|z^{-1}\gamma_k|,
	\]
	That is, for any $p\geq 1$,
	\begin{eqnarray*}
		P(|z^{-1}\alpha_k|>C)&\le& P(|\xi_k|>1/(2C))\le C\textrm{E}|\xi_k|^p,
	\end{eqnarray*}
	Since $|X_{jk}|\le\ep_n n^{1/4}$, choose $\eta=\ep_nn^{-1/4}$ in Lemma \ref{6.8}, for $p\geq\log n$,
	\begin{eqnarray*}
		\textrm{E}|\xi_k|^p&\le& C(n\epsilon_n^4n^{-1})^{-1}(v^{-1}\epsilon_n^2n^{-1/2})^p
		\le C\epsilon_n^{2p-4}\le C\epsilon_n^p.
	\end{eqnarray*}
	For any fixed $t>0$, when $n$ is large enough so that $\log \epsilon_n^{-1}>t+1$, it can be shown that
	\begin{eqnarray*}
		\textrm{E}|\xi_k|^p&\le&C\exp\{-p\log \epsilon_n^{-1}\}\le C\exp\{-p(t+1)\}\\
		&\le& C\exp\{-(t+1)\log n\}\\
		&=&Cn^{-t-1}=o(n^{-t}).
	\end{eqnarray*}
	This shows that Lemma \ref{agb} is true if we can prove equation \eqref{zbeta}.
	
	From (8.1.13) and (8.1.18) in \cite{Bai2010b}, we can see that
	\[
	|z^{-1}b_n|=\dfrac{1}{|z+\E s_n(z)|}=|-s(z+\de)|,
	\]
	which implies that $|z^{-1}b_n|\le 1$.
	Further from
	\[
	b_n-\beta_n=b_n\beta_nz^{-1}n^{-1}
	\left(\textrm{tr}\mathbf{A}^{-1}(z)-\textrm{Etr}\mathbf{A}^{-1}(z)\right)
	=b_n\beta_nz^{-1}(s_n(z)-\textrm{E}s_n(z)),
	\]
	we obtain
	\begin{eqnarray*}
		|z^{-1}\beta_n|&=&\dfrac{|b_n||z^{-1}b_n|}{|1+z^{-1}b_n(s_n(z)-\textrm{E}s_n(z))|}\\
		&\le&\dfrac{|z^{-1}b_n|}{1-|z^{-1}b_n||s_n(z)-\textrm{E}s_n(z)|}\\
		&\le&2|z^{-1}b_n|\le 2,
	\end{eqnarray*}
	if $|s_n(z)-\E s_n(z)|\le 1/2$.
	Therefore, by Lemma \ref{6.7}, for any $l> \max\{1,2t\}$ and $v\geq O(n^{-1/2})$, we have
	\begin{eqnarray*}
		P(|z^{-1}\beta_n|\geq 2)&\le& P(|s_n(z)-\E s_n(z)|\geq 1/2)\\
		&\le& C\E|s_n(z)-\E s_n(z)|^{2l} \\
		&\le& \dfrac{C}{n^{2l}v^{3l}}=o(n^{-t}).
	\end{eqnarray*}
	This completes the proof.
\end{proof}

\begin{lem}\label{asz}
	Under the conditions of Theorem \ref{thm1}, for any $l\geq 1$ and $z\in\cD$,
	there exists a constant $C>0$, such that
	\[
	\textrm{E}\left|z^{-1}(\gamma_k+zs(z))\right|^{2l}\le\dfrac{C}{(nv)^{2l}}.
	\]
\end{lem}
\begin{proof}
	By \eqref{beta_ga}, \eqref{zbeta}, \eqref{sz} and Lemma \ref{6.7}, \ref{Delta1}, it follows that
	\begin{eqnarray*}
		&&\textrm{E}\left|z^{-1}(\gamma_k+zs(z))\right|^{2l}
		\le C\E\left|z^{-1}(\gamma_k-\beta_n)\right|^{2l} +C\E\left|z^{-1}(\beta_n+zs(z))\right|^{2l}\\
		&&=C\textrm{E}\left|z^{-2}\gamma_k\beta_n n^{-1}(\textrm{tr}\mathbf{A}_k^{-1}(z)+z^{-1}-\textrm{tr}\mathbf{A}^{-1}(z))\right|^{2l}
		+C\textrm{E}\left|z^{-1}\beta_n-\dfrac{1}{z+s(z)}\right|^{2l}\\
		&&\le\dfrac{C}{(nv)^{2l}}+C\textrm{E}\left|\dfrac{z^{-1}\beta_n}{z+s(z)} \left(\dfrac{1}{n}\textrm{tr}\mathbf{A}^{-1}(z)-s(z)\right)\right|^{2l}\\
		&&\le\dfrac{C}{(nv)^{2l}}+C\textrm{E}\left|s_n(z)-\textrm{E}s_n(z)\right|^{2l}+C\left|\textrm{E}s_n(z)-s(z)\right|^{2l}\\
		&&\le\dfrac{C}{(nv)^{2l}}+\dfrac{C}{n^{2l}v^{3l}}+\dfrac{C}{n^l}\\
		&&\le\dfrac{C}{(nv)^{2l}}.
	\end{eqnarray*}
\end{proof}

\begin{lem}\label{mem}
	Under the conditions of Theorem \ref{thm1}, for $z\in\cD$, we have
	\begin{eqnarray*}
		\textrm{E}|s_n^H(z)-\textrm{E}s_n^H(z)|^{2l}\le \dfrac{C}{n^lv^{2l}}\left(1+\dfrac{\Delta^H}{v}\right)^{2l} +\dfrac{C}{n^{l/2+2}v^l}\left(1+\dfrac{\Delta^H}{v}\right)^l.
	\end{eqnarray*}
\end{lem}

\begin{proof}
	Write $\fx_n^*\fA^{-1}(z)\fx_n-\E\fx_n^*\fA^{-1}(z)\fx_n$ as the sum of a martingale difference sequence.
	Denote $\E_k(\cdot)$ as the conditional expectation given $\{X_{ij},k<i,j\le n\}$.
	\begin{eqnarray*}
		\mathbf{x}_n^*\mathbf{A}^{-1}(z)\mathbf{x}_n-\E\mathbf{x}_n^*\mathbf{A}^{-1}(z)\mathbf{x}_n
		&=&\sum_{k=1}^n(\E_{k-1}-\E_k)\fx_n^*\fA^{-1}(z)\fx_n\\
		&=&\sum_{k=1}^n(\textrm{E}_{k-1}-\textrm{E}_k)\mathbf{x}_n^*(\mathbf{A}^{-1}(z)-\mathbf{A}_k^{-1}(z))\mathbf{x}_n\\
		&=&\sum_{k=1}^n(\textrm{E}_{k-1}-\textrm{E}_k)\mathbf{x}_n^*\mathbf{A}^{-1}(z)(\mathbf{A}_k(z)-\mathbf{A}(z))\mathbf{A}_k^{-1}(z)\mathbf{x}_n\\
		&=&-\sum_{k=1}^n(\textrm{E}_{k-1}-\textrm{E}_k)\mathbf{x}_n^*\mathbf{A}^{-1}(z) (\mathbf{w}_k\mathbf{e}_k^*+\mathbf{e}_k\mathbf{w}_k^*)\mathbf{A}_k^{-1}(z)\mathbf{x}_n\\
		&:=&\phi_{n1}+\phi_{n2}.
	\end{eqnarray*}
	By \eqref{Awk} and \eqref{Ak1}, we further decompose $\phi_{n1}$ as
	\begin{eqnarray*}
		\phi_{n1}&=&-\sum_{k=1}^n(\textrm{E}_{k-1}-\textrm{E}_k)\mathbf{x}_n^*\mathbf{A}^{-1}(z)\mathbf{w}_k\mathbf{e}_k^*\mathbf{A}_k^{-1}(z)\mathbf{x}_n\\
		&=&z^{-1}\sum_{k=1}^n(\textrm{E}_{k-1}-\textrm{E}_k)x_k\mathbf{x}_n^*\mathbf{A}_{k1}^{-1}(z)\mathbf{w}_k\alpha_k\\
		&=&z^{-1}\sum_{k=1}^n(\textrm{E}_{k-1}-\textrm{E}_k)x_k \left(\mathbf{x}_n^*\mathbf{A}_k^{-1}(z)\mathbf{w}_k\alpha_k
		-\mathbf{x}_n^*\mathbf{A}_k^{-1}(z)\mathbf{e}_k\mathbf{w}_k^*\mathbf{A}_k^{-1}(z)\mathbf{w}_k\alpha_k\right)\\
		&=&z^{-1}\sum_{k=1}^n(\textrm{E}_{k-1}-\textrm{E}_k)x_k\mathbf{x}_n^*\mathbf{A}_k^{-1}(z)\mathbf{w}_k\alpha_k
		+z^{-2}\sum_{k=1}^n(\textrm{E}_{k-1}-\textrm{E}_k)|x_k|^2\mathbf{w}_k^*\mathbf{A}_k^{-1}(z)\mathbf{w}_k\alpha_k\\
		&=&z^{-1}\sum_{k=1}^n(\textrm{E}_{k-1}-\textrm{E}_k)x_k\mathbf{x}_n^*\mathbf{A}_k^{-1}(z)\mathbf{w}_k\alpha_k
		-z^{-1}\sum_{k=1}^n(\textrm{E}_{k-1}-\textrm{E}_k)|x_k|^2\alpha_k\\
		&:=&\phi_{n11}+\phi_{n12}.
	\end{eqnarray*}
	Similar to $\phi_{n1}$, we obtain
	\begin{eqnarray*}
		\phi_{n2}&=&-\sum_{k=1}^n(\textrm{E}_{k-1}-\textrm{E}_k)\mathbf{x}_n^*\mathbf{A}^{-1}(z)\mathbf{e}_k\mathbf{w}_k^*\mathbf{A}_k^{-1}(z)\mathbf{x}_n\\
		&=&-\sum_{k=1}^n(\textrm{E}_{k-1}-\textrm{E}_k)\mathbf{x}_n^*\mathbf{A}_{k2}^{-1}(z)\mathbf{e}_k\mathbf{w}_k^*\mathbf{A}_k^{-1}(z)\mathbf{x}_n\alpha_k\\
		&=&-\sum_{k=1}^n(\textrm{E}_{k-1}-\textrm{E}_k)\mathbf{x}_n^* (\mathbf{A}_k^{-1}(z)-\mathbf{A}_k^{-1}(z)\mathbf{w}_k\mathbf{e}_k^*\mathbf{A}_k^{-1}(z)) \mathbf{e}_k\mathbf{w}_k^*\mathbf{A}_k^{-1}(z)\mathbf{x}_n\alpha_k\\
		&=&-\sum_{k=1}^n(\textrm{E}_{k-1}-\textrm{E}_k)\mathbf{x}_n^*\mathbf{A}_k^{-1}(z)\mathbf{e}_k\mathbf{w}_k^*\mathbf{A}_k^{-1}(z)\mathbf{x}_n\alpha_k\\
		&&+\sum_{k=1}^n(\textrm{E}_{k-1}-\textrm{E}_k)\mathbf{x}_n^*\mathbf{A}_k^{-1}(z)\mathbf{w}_k\mathbf{e}_k^*\mathbf{A}_k^{-1}(z)\mathbf{e}_k\mathbf{w}_k^*\mathbf{A}_k^{-1}(z)\mathbf{x}_n\alpha_k\\
		&=&\sum_{k=1}^n(\textrm{E}_{k-1}-\textrm{E}_k)z^{-1}\bar{x}_k\mathbf{w}_k^*\mathbf{A}_k^{-1}(z)\mathbf{x}_n\alpha_k\\
		&&-\sum_{k=1}^n(\textrm{E}_{k-1}-\textrm{E}_k)z^{-1}\mathbf{x}_n^*\mathbf{A}_k^{-1}(z)\mathbf{w}_k\mathbf{w}_k^*\mathbf{A}_k^{-1}(z)\mathbf{x}_n\alpha_k\\
		&:=&\phi_{n21}+\phi_{n22}.
	\end{eqnarray*}
	Notice that $\phi_{nij}$, $i,j=1,2$ are martingale difference sequence,
	for any $l\geq 1$, we conclude from Lemma \ref{Burkholder}(b) and \ref{xtrx} that
	\begin{eqnarray*}
		\E|\phi_{n11}|^{2l}&\le& C\E
		\(\sum_{k=1}^n\E_k\lj(\E_{k-1}-\E_k)x_k \fx_n^*\fA_k^{-1}(z)\fw_kz^{-1}\al_k\rj^2\)^l\\
		&&+C\sum_{k=1}^n\E\lj(\E_{k-1}-\E_k)x_k \fx_n^*\fA_k^{-1}(z)\fw_kz^{-1}\al_k\rj^{2l}\\
		&\le& C\E\(\sum_{k=1}^n|x_k|^2\E_k|\fx_n^*\fA_k^{-1}(z)\fw_k|^2\)^l
		+C\sum_{k=1}^n|x_k|^{2l}\E|\fx_n^*\fA_k^{-1}(z)\fw_k|^{2l} \\
		&\le&\dfrac{C}{n^lv^l}\E\(\Im(s_n^H(z))\)^l+\dfrac{C}{n^{l/2+2}v^l}\E\(\Im(s_n^H(z))\)^l.
	\end{eqnarray*}
	As in dealing with $\E|\phi_{n11}|^{2l}$, similarly we obtain
	\[
	\E|\phi_{n21}|^{2l}\le\dfrac{C}{n^lv^l}\E\(\Im(s_n^H(z))\)^l+\dfrac{C}{n^{l/2+2}v^l}\E\(\Im(s_n^H(z))\)^l.
	\]
	
	Since $\al_k=\ga_k-z^{-1}\al_k\ga_k\xi_k$ and $(\E_{k-1}-\E_k)\ga_k=0$, we have
	\[
	\phi_{n12}=z^{-2}\sum_{k=1}^n(\E_{k-1}-\E_k)|x_k|^2\al_k\ga_k\xi_k.
	\]
	Again, by Lemma \ref{Burkholder}(b) and \ref{agb}, it follows that
	\begin{eqnarray*}
		\E|\phi_{n12}|^{2l}&\le& C\E
		\(\sum_{k=1}^n|x_k|^4\E_k|\xi_k|^2\)^l+C\sum_{k=1}^n|x_k|^{4l}\E|\xi_k|^{2l}\\
		&\le&\dfrac{C}{n^lv^l}(1+\De/v)^l+\dfrac{C}{n^{l+1}v^{2l-1}}(1+\De/v).
	\end{eqnarray*}
	
	Now, we turn to term $\phi_{n22}$.
	Substitute $\al_k=\ga_k-z^{-1}\al_k\ga_k\xi_k$ into $\phi_{n22}$, and note that
	\begin{eqnarray*}
		&&(\textrm{E}_{k-1}-\textrm{E}_k)\mathbf{x}_n^*\mathbf{A}_k^{-1}(z)\mathbf{w}_k\mathbf{w}_k^*\mathbf{A}_k^{-1}(z)\mathbf{x}_n\gamma_k\\
		&=&\textrm{E}_{k-1}\gamma_k\left(
		\mathbf{w}_k^*\mathbf{A}_k^{-1}(z)\mathbf{x}_n\mathbf{x}_n^*\mathbf{A}_k^{-1}(z)\mathbf{w}_k -n^{-1}\mathbf{x}_n^*\mathbf{A}_k^{-1}(z)
		(\mathbf{I}-\mathbf{e}_k\mathbf{e}_k^*)\mathbf{A}_k^{-1}(z)\mathbf{x}_n
		\right)\\
		&=&\textrm{E}_{k-1}\gamma_k\left(
		\mathbf{w}_k^*\mathbf{A}_k^{-1}(z)\mathbf{x}_n\mathbf{x}_n^*\mathbf{A}_k^{-1}(z)\mathbf{w}_k
		-n^{-1}(\mathbf{x}_n^*\mathbf{A}_k^{-2}(z)\mathbf{x}_n -\mathbf{x}_n^*\mathbf{A}_k^{-1}(z)\mathbf{e}_k\mathbf{e}_k^*\mathbf{A}_k^{-1}(z)\mathbf{x}_n)
		\right)\\
		&=&\textrm{E}_{k-1}\gamma_k\left(\mathbf{w}_k^*\mathbf{A}_k^{-1}(z)\mathbf{x}_n\mathbf{x}_n^*\mathbf{A}_k^{-1}(z)\mathbf{w}_k
		-n^{-1}(\mathbf{x}_n^*\mathbf{A}_k^{-2}(z)\mathbf{x}_n-z^{-2}|x_k|^2)
		\right)\\
		&=&\textrm{E}_{k-1}\gamma_k\eta_k
	\end{eqnarray*}
	We further decompose $\phi_{n22}$ as
	\begin{eqnarray*}
		\phi_{n22}&=&-\sum_{k=1}^nz^{-1}\textrm{E}_{k-1}\gamma_k\eta_k\\
		&&+\sum_{k=1}^n(\textrm{E}_{k-1}-\textrm{E}_k)z^{-2} \left(\mathbf{w}_k^*\mathbf{A}_k^{-1}(z)\mathbf{x}_n\mathbf{x}_n^*\mathbf{A}_k^{-1}(z)\mathbf{w}_k\right)
		\alpha_k\gamma_k\xi_k\\
		&&-\bigg(\sum_{k=1}^n(\textrm{E}_{k-1}-\textrm{E}_k)z^{-2}\dfrac{1}{n} \left(\mathbf{x}_n^*\mathbf{A}_k^{-2}(z)\mathbf{x}_n-z^{-2}|x_k|^2\right)
		\alpha_k\gamma_k\xi_k\\
		&&-\sum_{k=1}^n(\textrm{E}_{k-1}-\textrm{E}_k)z^{-2}\dfrac{1}{n} \left(\mathbf{x}_n^*\mathbf{A}_k^{-2}(z)\mathbf{x}_n-z^{-2}|x_k|^2\right)
		\alpha_k\gamma_k\xi_k\bigg)\\
		&=&-\sum_{k=1}^nz^{-1}\textrm{E}_{k-1}\gamma_k\eta_k\\
		&&+\sum_{k=1}^n(\textrm{E}_{k-1}-\textrm{E}_k)z^{-2}\eta_k\alpha_k\gamma_k\xi_k\\
		&&+\sum_{k=1}^n(\textrm{E}_{k-1}-\textrm{E}_k)z^{-2}\dfrac{1}{n} \left(\mathbf{x}_n^*\mathbf{A}_k^{-2}(z)\mathbf{x}_n-z^{-2}|x_k|^2\right)
		\alpha_k\gamma_k\xi_k\\
		&:=&\phi_{n221}+\phi_{n222}+\phi_{n223}.
	\end{eqnarray*}
	By Lemma \ref{xtrx}, it is easy to verify that
	\[
	\E|\eta_k|^{2l}\le\dfrac{C}{n^{2l}v^{2l}}\E\(\Im(s_n^H(z))\)^{2l}
	+\dfrac{C}{n^{l+2}v^{2l}}\E\(\Im(s_n^H(z))\)^{2l}.
	\]
	Therefore, by Lemma \ref{Burkholder}(b), it leads to
	\[
	\E|\phi_{n22j}|^{2l}\le\dfrac{C}{n^lv^{2l}}\E\(\Im(s_n^H(z))\)^{2l}, ~~~{\rm for}~j=1,2,3.
	\]
	
	Combined the above results together, we obtain
	\[
	\E|s_n^H(z)-\E s_n^H(z)|^{2l}\le\dfrac{C}{n^{l/2+2}v^l}\E\(\Im(s_n^H(z))\)^l
	+\dfrac{C}{n^lv^{2l}}\E\(\Im(s_n^H(z))\)^{2l}.
	\]
	Note that
	\[
	\E\(\Im(s_n^H(z))\)^{2l}
	\le C\E|s_n^H(z)-\E s_n^H(z)|^{2l}+C(1+\De^H/v)^{2l}.
	\]
	For $z\in\cD$, that is $v\geq C_0n^{-1/2}$, we can choose $C_0$ large enough, such that
	$\dfrac{C}{n^lv^{2l}}\le 2/3$. Therefore,
	\begin{eqnarray}\label{eses}
	\E|s_n^H(z)-\E s_n^H(z)|^{2l}
	\le \dfrac{C}{n^{l/2+2}v^l}\E\(\Im(s_n^H(z))\)^l+\dfrac{C}{n^lv^{2l}}(1+\De/v)^{2l}.
	\end{eqnarray}
	Next, we use induction to treat the term $\E\(\Im(s_n^H(z))\)^l$ for $l>1$.
	
	When $1/2<l<1$, applying Lemma \ref{Burkholder}(a) to $\E|\phi_{nij}|^{2l}$, we have
	\begin{eqnarray*}
		\E|s_n^H(z)-\E s_n^H(z)|^{2l} &\le& C\sum_{i,j=1}^2\E|\phi_{nij}|^{2l}\\
		&\le&\dfrac{C}{n^lv^{2l}}(1+\De^H/v)^l ~<~C.
	\end{eqnarray*}
	This shows that for $1<p<2$,
	\begin{eqnarray*}
		\E\(\Im(s_n^H(z))\)^p &\le& 2\E |s_n^H(z)-\E s_n^H(z)|^p+2|\E s_n^H(z)|^p\\
		&\le& C\(1+(1+\De^H/v)^p\)
	\end{eqnarray*}
	Thus the lemma follows for $1<l<2$ by substituting the result of $\E\(\Im(s_n^H(z))\)^p$
	into \eqref{eses}.
	Therefore, the lemma holds by using induction and \eqref{eses}.
	This completes the proof of lemma \ref{mem}.
\end{proof}

\begin{lem}\label{tail}
	Under the conditions in Theorem \ref{thm1}, for any fixed $t>0$,
	\[
	\int_3^{\infty} \lj\E H^{\fW_n}(x)-F(x)\rj dx=o(n^{-t}).
	\]
\end{lem}
\begin{proof}
	For any fixed $t>0$ and $x>0$, by Lemma \ref{outside}, it follows that
	\[
	P\(\la_{\max}(\fW_n)\geq 3+x\)\le Cn^{-(t+1)}(3+x)^{-2}.
	\]
	Note that
	\[
	1-H^{\fW_n}(x)\le \indic(\la_{\max}(\fW_n)\geq x), ~~~~ {\rm for}~ x\geq 0.
	\]
	Therefore, we have
	\begin{eqnarray*}
		\int_3^{\infty}\lj\E H^{\fW_n}(x)-F(x)\rj dx
		&\le& \int_3^{\infty}P\(\la_{\max}(\fW_n)\geq x\)dx \\
		&\le& \int_0^{\infty}Cn^{-(t+1)}(3+x)^{-2}dx \\
		&=&o(n^{-t}).
	\end{eqnarray*}
\end{proof}

\section{Appendix II}\label{Appendix2}
\begin{lem}(\textbf{Burkholder Inequalities}).
	Let $\{X_k\}$ be a complex martingale difference sequence with respect to the
	increasing $\sigma$-field $\{\mathcal{F}_k\}$, and let $\mbox{\em E}_k$ denote the conditional expectation with respect to $\mathcal{F}_k$. Then we have\\
	\noindent (a) for $p>1$,
	\begin{eqnarray*}
		\mbox{\em E} \bigg|\sum_{k=1}^n X_k \bigg|^p \leq K_p \mbox{\em E} \bigg(\sum_{k=1}^n
		|X_k|^2 \bigg)^{p/2};
	\end{eqnarray*}
	\noindent (b) for $p \geq 2$,
	\begin{eqnarray*}
		\mbox{\em E}\bigg|\sum_{k=1}^n X_k\bigg|^p \leq K_p \bigg(\mbox{\em E}\big(\sum_{k=1}^n
		\mbox{\em E}_{k-1}|X_k|^2 \big)^{p/2}+\mbox{\em E}\sum_{k=1}^n |X_k|^p\bigg).
	\end{eqnarray*}
	\label{Burkholder}
	where $K_p$ is a constant which depends upon $p$ only.
\end{lem}

\begin{lem}(Lemma 2.6 of \cite{MR1345534})
	Let $z\in \mathbb{C}^+$ with $v=\Im z$, $\mathbf{A}$ and $\mathbf{B}$ Hermitian and $\mathbf{r}\in \mathbb{C}^n$. Then
	$$\left|\textrm{tr}\left((\mathbf{B}-z\mathbf{I})^{-1} -(\mathbf{B}+\mathbf{rr}^*-z\mathbf{I})^{-1}\right)\mathbf{A}\right| =\left|\frac{\mathbf{r}^*(\mathbf{B}-z\mathbf{I})^{-1}\mathbf{A} (\mathbf{B}-z\mathbf{I})^{-1}\mathbf{r}}{1+\mathbf{r}^*(\mathbf{B}-z\mathbf{I})^{-1}\mathbf{r}}\right|\le \frac{\|\mathbf{A}\|}{v}.$$
	\label{rr}
\end{lem}

\begin{lem} ((1.15) in \cite{MR2040792})
	Let $\mathbf{Y}=(Y_1,\cdots,Y_n)$, where $Y_i$'s are i.i.d. complex random variables with mean zero and variance 1. Let $\mathbf{A}=(a_{ij})_{n\times n}$ and $\mathbf{B}=(b_{ij})_{n\times n}$ be complex matrices. Then the following identity holds:
	\begin{eqnarray*}
		&&\textrm{E}(\mathbf{Y}^*\mathbf{AY}-\textrm{tr}\mathbf{A}) (\mathbf{Y}^*\mathbf{BY}-\textrm{tr}\mathbf{B})\\
		&& \qquad\qquad=(\textrm{E}|Y_1|^4-|\textrm{E}Y_1^2|^2-2)\sum_{i=1}^na_{ii}b_{ii}+ |\textrm{E}Y_1^2|^2\textrm{tr}\mathbf{AB}^T+\textrm{tr}\mathbf{AB}.
	\end{eqnarray*}
	\label{xabx}
\end{lem}

\begin{lem}(Lemma 2.7 of \cite{MR1617051})
	For $\mathbf{X}=(X_1,\cdots,X_n)^{\prime}$ with i.i.d. standardized real or complex entries such that $\textrm{E}X_i=0$ and $\textrm{E}|X_i|^2=1$, and for $\mathbf{C}$ an $n\times n$ complex matrix, we have, for any $p\geq 2$,
	\[
	\textrm{E}|\mathbf{X}^*\mathbf{CX}-\textrm{tr}\mathbf{C}|^p\le K_p\left[\left(\textrm{E}|X_i|^4\textrm{tr}\mathbf{CC}^*\right)^{p/2} +\textrm{E}|X_i|^{2p}\textrm{tr}(\mathbf{CC}^*)^{p/2}\right].
	\]
	\label{xtrx}
	where $K_p$ is a constant which depends upon $p$ only.
\end{lem}

\begin{lem}((5.1.16) of \cite{Bai2010b})
	Suppose that the entries of the matrix $\mathbf{W}_n=\dfrac{1}{\sqrt{n}}\left(X_{ij}\right)$ are independent (but depend on $n$) and satisfy \\
	\noindent (1) $\mbox{\em E}X_{jk}=0$,\\
	(2) $\textrm{E}|X_{jk}|^2\leq \sigma^2,$\\
	(3) $\mbox{\em E}|X_{jk}|^l\leq b(\delta_n\sqrt{n})^{l-3}$ for all $l\geq 3$,\\
	where $\delta_n\to 0$ and $b>0$. Then, for fixed $\epsilon>0$ and $x>0$,
	\begin{eqnarray*}
		\mbox{\em P}\big(\lambda_{\mbox{max}}(\mathbf{W}_n)\geq 2\sigma+\epsilon+x\big)\leq o(n^{-l}(2\sigma+\epsilon+x)^{-2}).
	\end{eqnarray*}
	\label{outside}
\end{lem}

\begin{lem}(Theorem 8.6 in \cite{Bai2010b})
	Under assumptions in Theorem \ref{thm1}, we have
	\[
	\Delta=\|\textrm{E}F^{W_n}-F\|=O(n^{-1/2}).
	\]
	\label{Delta1}
\end{lem}

\begin{lem}(Lemma 8.20 in \cite{Bai2010b})
	If $|z|<C$, $v\geq O(n^{-1/2})$ and $l \geq 1$, then
	\begin{eqnarray*}
		\textrm{E}|s_n(z)-\textrm{E}s_n(z)|^{2l}\le Cn^{-2l}v^{-4l}(\Delta+v)^l,
	\end{eqnarray*}
	where $\Delta=\|\textrm{E}F^{W_n}-F\|$.
	\label{6.7}
\end{lem}

\begin{lem}(Lemma 9.1 in \cite{Bai2010b})
	Suppose that $X_i$, $i=1,\cdots, n$, are independent, with $\mbox{\em E}X_i=0$, $\mbox{\em E}|X_i|^2=1$, $\sup\mbox{\em E}|X_i|^4=\nu < \infty$
	and $|X_i|\leq \eta \sqrt{n}$ with $\eta>0$. Assume that $\mathbf{A}$ is a complex matrix. Then for any given $p$ such that $2\leq p \leq b \log(n\nu^{-1}\eta^4 )$ and
	$b>1$, we have
	\begin{eqnarray*}
		\mbox{\em E}|\mathbf{\alpha}^* \mathbf{A} \mathbf{\alpha} -\mbox{\em tr} (\mathbf{A})|^p \leq \nu n^p(n\eta^4)^{-1} (40b^2\|\mathbf{A}\|\eta^2)^p,
	\end{eqnarray*}
	where $\mathbf{\alpha}=(X_1,\cdots, X_n)^T$.
	\label{6.8}
\end{lem}

\section{Truncation Part}\label{truncation}
We shall prove that the underlying random variables satisfy $X_{ii}=0$ and for
$i\neq j$, $|X_{ij}|\le \epsilon_nn^{1/4}$, $\textrm{E}X_{ij}=0$ and
$\var(X_{ij})=1$, where $\epsilon_n$ is a sequence decreasing to 0 and
$\epsilon_nn^{1/4}$ increasing to infinity.
\par
\subsection{Truncation}
Under finite 10-th moment condition, we can choose $\epsilon_n$ decreasing to 0 and $\epsilon_nn^{1/4}$ increasing to infinity as $n\rightarrow\infty$, such that
\[
\epsilon_n^{-10}\textrm{E}|X_{12}|^{10}\textrm{I}\left(|X_{12}|>\epsilon_nn^{1/4}\right)\to 0.
\]
Let $\widehat{\mathbf{X}}_n$ denote the truncated matrix whose entry on the $i$-th row and $j$-th column is $X_{ij}\textrm{I}\left(|X_{ij}|\le\epsilon_nn^{1/4}\right)$, $i=1,\cdots,n$, $j=1,\cdots,n$. Define $\widehat{\mathbf{W}}_n=\dfrac{1}{\sqrt{n}}\widehat{\mathbf{X}}_n$.
\par
\textbf{Step 1(a). Truncation for Theorem {\ref{thm1}}.}
\begin{eqnarray*}
	P\left(\mathbf{W}_n\neq \widehat{\mathbf{W}}_n\right)&\le&n^2P\left(|X_{ij}|>\epsilon_nn^{1/4}\right)\\
	&\le&n^2\epsilon_n^{-10}n^{-5/2}\textrm{E}|X_{12}|^{10}\textrm{I}\left(|X_{12}|>\epsilon_nn^{1/4}\right)\\
	&=&o(n^{-1/2}).
\end{eqnarray*}
\textbf{Step 1(b). Truncation for Theorem \ref{thm2} and Theorem \ref{thm3}.}
\begin{eqnarray*}
	P\left(\mathbf{W}_n\neq\widehat{\mathbf{W}}_n, i.o.\right)&=&\lim_{k\to\infty}
	P\left(\cup_{n=k}^\infty \cup_{i=1}^n \cup_{j=1}^n |X_{ij}|>\epsilon_nn^{1/4}\right)\\
	&=&\lim_{k\to\infty}P\left(\cup_{t=k}^\infty \cup_{n\in[2^t,2^{t+1})} \cup_{i=1}^n \cup_{j=1}^n|X_{ij}|>\epsilon_nn^{1/4}\right)\\
	&\le&\lim_{k\to\infty}\sum_{t=k}^\infty P\left(\cup_{i=1}^{2^{t+1}} \cup_{j=1}^{2^{t+1}}|X_{ij}|>\epsilon_{2^t}2^{t/4}\right)\\
	&\le&\lim_{k\to\infty}\sum_{t=k}^\infty (2^{t+1})^2 P\left(|X_{12}|>\epsilon_{2^t} 2^{t/4}\right)\\
	&\le&\lim_{k\to\infty}\sum_{t=k}^\infty \sum_{l=t}^\infty 4^t P\left(\epsilon_{2^l}2^{l/4}<|X_{12}|\le\epsilon_{2^{l+1}}2^{(l+1)/4}\right)\\
	&=&\lim_{k\to\infty}\sum_{l=k}^\infty \sum_{t=k}^l 4^t P\left(\epsilon_{2^l}2^{l/4}<|X_{12}|\le\epsilon_{2^{l+1}}2^{(l+1)/4}\right)\\
	&\le&\lim_{k\to\infty}\sum_{l=k}^\infty\epsilon_{2^l}^{-8}\textrm{E}|X_{12}|^8\textrm{I}\left(\epsilon_{2^l}2^{l/4}<|X_{12}|\le\epsilon_{2^{l+1}}2^{(l+1)/4}\right)\\
	&=&0.
\end{eqnarray*}

\subsection{Removing diagonal entries}
Based on step 1, we can assume that $\mathbf{W}_n=\dfrac{1}{\sqrt{n}}\left(X_{ij}\textrm{I}\left(|X_{ij}|\le\epsilon_nn^{1/4}\right)\right)_{i,j=1}^n$.
Let $\mathbf{W}_n^0=\dfrac{1}{\sqrt{n}}\left(X_{ij}\textrm{I}\left(|X_{ij}|\le\epsilon_nn^{1/4}\right)\left(1-\delta_{ij}\right)\right)_{i,j=1}^n$.
Here $\delta_{ij}=1$ or $0$ according to $i=j$, or $i\neq j$.
\par
\textbf{Step 2(a). Removing diagonal entries for Theorem \ref{thm1}.}
\begin{eqnarray*}
	&&\left\|\textrm{E}\left(\mathbf{W}_n-\mathbf{W}_n^0\right)\right\|\\
	&=&\dfrac{1}{\sqrt{n}}\left\|\textrm{E}\left(\textrm{diag}\left(X_{11}\textrm{I}\left(|X_{11}|\le\epsilon_nn^{1/4}\right),\cdots,
	X_{nn}\textrm{I}\left(|X_{nn}|\le\epsilon_nn^{1/4}\right)\right)\right)\right\|\\
	&=&\dfrac{1}{\sqrt{n}}\left\|\textrm{E}\left(\textrm{diag}\left(X_{11}\textrm{I}\left(|X_{11}|>\epsilon_nn^{1/4}\right),\cdots,
	X_{nn}\textrm{I}\left(|X_{nn}|>\epsilon_nn^{1/4}\right)\right)\right)\right\|\\
	&=&\dfrac{1}{\sqrt{n}}\textrm{E}|X_{11}|\textrm{I}\left(|X_{11}|>\epsilon_nn^{1/4}\right)\\
	&=&o(n^{-1/2}).
\end{eqnarray*}

\textbf{Step 2(b). Removing diagonal entries for Theorem \ref{thm2} and Theorem \ref{thm3}.}
\begin{eqnarray*}
	\left\|\mathbf{W}_n-\mathbf{W}_n^0\right\|
	&=&\dfrac{1}{\sqrt{n}}\left\|\textrm{diag}\left(X_{11}\textrm{I}\left(|X_{11}|\le\epsilon_nn^{1/4}\right),\cdots,
	X_{nn}\textrm{I}\left(|X_{nn}|\le\epsilon_nn^{1/4}\right)\right)\right\|\\
	&=&\dfrac{1}{\sqrt{n}}\epsilon_nn^{1/4}
	=o(n^{-1/4}).
\end{eqnarray*}

\subsection{Centralization}
Suppose that $z\in\cD$ and denote
$\widehat{\mathbf{W}}_n=\dfrac{1}{\sqrt{n}}\left(\widehat{X}_{ij}\right)$, where
\[
\widehat{X}_{ij}=X_{ij}\textrm{I}\left(|X_{ij}|\le\epsilon_nn^{1/4}\right)-\textrm{E}X_{ij}\textrm{I}\left(|X_{ij}|\le\epsilon_nn^{1/4}\right).
\]
\textbf{Centralization for Theorem \ref{thm1}.}
\begin{eqnarray*}
	\textrm{E}|s_{H^{W_n}}(z)-s_{H^{\widehat{W}_n}}(z)|
	&=&\textrm{E}\left|\mathbf{x}_n^*\left(\mathbf{W}_n-z\mathbf{I}\right)^{-1}\mathbf{x}_n-\mathbf{x}_n^*(\widehat{\mathbf{W}}_n-z\mathbf{I})^{-1}\mathbf{x}_n\right|\\
	&\le&\dfrac{C}{\sqrt{n}}\bigg[\left|\textrm{E}X_{12}\textrm{I}\left(|X_{ij}|\le\epsilon_nn^{1/4}\right) \right|
	\left\|\mathbf{1}_n\right\| \left\|\mathbf{1}_n^{\prime}\right\|\\
	&& \big(\textrm{E}\|\mathbf{x}_n^*(\mathbf{W}_n-z\mathbf{I})^{-1}\|^2\big)^{1/2}
	\left(\textrm{E}\|(\widehat{\mathbf{W}}_n-z\mathbf{I})^{-1}\mathbf{x}_n\|^2\right)^{1/2}\bigg]\\
	&\le& C\sqrt{n}v^{-1}\left|\textrm{E}X_{12}\textrm{I}\left(|X_{ij}|>\epsilon_nn^{1/4}\right)\right|\\
	&=&o(n^{-1/2}).
\end{eqnarray*}

\textbf{Centralization for Theorem \ref{thm2} and Theorem \ref{thm3}.}
\begin{eqnarray*}
	|s_{H^{W_n}}(z)-s_{H^{\widehat{W}_n}}(z)|
	&=&\left|\mathbf{x}_n^*\left(\mathbf{W}_n-z\mathbf{I}\right)^{-1}\mathbf{x}_n-\mathbf{x}_n^*(\widehat{\mathbf{W}}_n-z\mathbf{I})^{-1}\mathbf{x}_n\right|\\
	&\le&\left\|\left(\mathbf{W}_n-z\mathbf{I}\right)^{-1}\right\|  \left\|\mathbf{W}_n-\widehat{\mathbf{W}}_n\right\|
	\left\|(\widehat{\mathbf{W}}_n-z\mathbf{I})^{-1}\right\|\\
	&\le&\dfrac{C}{v^2}\left\|\mathbf{W}_n-\widehat{\mathbf{W}}_n\right\|\\
	&=&\dfrac{C}{v^2}\dfrac{1}{\sqrt{n}}\left|\textrm{E}X_{12}\textrm{I}\left(|X_{ij}|\le\epsilon_nn^{1/4}\right) \right|
	\left\|\mathbf{1}_n\right\| \left\|\mathbf{1}_n^{\prime}\right\|\\
	&\le&C\sqrt{n}v^{-2}\left|\textrm{E}X_{12}\textrm{I}\left(|X_{ij}|>\epsilon_nn^{1/4}\right) \right| \\
	&=&o(n^{-1/4}).
\end{eqnarray*}

\subsection{Rescaling}
The rescaling procedures for the three theorems are exactly the same, and only 8-th moment is required. Thus we treat them uniformly. Write
$\mathbf{W}_n=\dfrac{1}{\sqrt{n}}\left(\widehat{X}_{ij}\right)$, where $\widehat{X}_{ij}=X_{ij}\textrm{I}\left(|X_{ij}|\le\epsilon_nn^{1/4}\right)-\textrm{E}X_{ij}\textrm{I}\left(|X_{ij}|\le\epsilon_nn^{1/4}\right)$.
And $\widetilde{\mathbf{W}}_n=\dfrac{1}{\sqrt{n}}\left(Y_{ij}\right)$, where $Y_{ij}=X_{ij}/\sigma_1$ and
$\sigma_1^2=\textrm{E}\bigg|X_{12}\textrm{I}\left(|X_{12}|\le\epsilon_nn^{1/4}\right)-\textrm{E}X_{12}\textrm{I}\left(|X_{12}|\le\epsilon_nn^{1/4}\right)\bigg|^2$.
Notice that $\sigma_1\le 1$, and $\sigma_1$ tends to 1 as $n$ goes to $\infty$.
For $z\in\cD$, we obtain
\begin{eqnarray*}
	|s_{H^{W_n}}(z)-s_{H^{\widetilde{W}_n}}(z)|&=&
	\left|x_n^*(W_n-zI)^{-1}(W_n-\widetilde{W}_n)(\widetilde{W}_n-zI)^{-1}x_n\right|\\
	&\le&\dfrac{1}{v^2}\|(1-\sigma_1^{-1})W_n\|\quad (\textrm{by lemma } \ref{outside})\\
	&\le&\dfrac{C}{v^2}(1-\sigma_1)\\
	&\le&\dfrac{C}{v^2}(1-\sigma_1^2)\\
	&\le&Cv^{-2}\textrm{E}|X_{12}|^2\textrm{I}\left(|X_{12}|>\epsilon_nn^{1/4}\right)\\
	&=&o(n^{-1/2}).
\end{eqnarray*}



\end{document}